\documentclass[a4paper,reqno]{amsart}
\pdfoutput=1

\usepackage{amsmath,amsfonts,amssymb,amsthm,stmaryrd,bbm,graphicx,mathtools,enumerate}
\usepackage{mathrsfs}
\usepackage[T1]{fontenc} 
\usepackage[latin1]{inputenc}
\usepackage[english]{babel}
\usepackage[tracking,spacing,kerning,babel]{microtype}
\usepackage{wasysym} 
\usepackage{color}
\usepackage{xcolor}
    	\usepackage{framed} 
\usepackage{wrapfig} 

\usepackage[final]{hyperref}   
\hypersetup{
    linktoc=page,
    linkcolor=red,          
    citecolor=blue,        
    filecolor=blue,      
    urlcolor=cyan,
   colorlinks=true           
}

\usepackage{lipsum} 

\usepackage{minitoc}
\usepackage{multicol}

\newtheoremstyle{mine}
{\baselineskip}
{\baselineskip}
{\itshape}
{
}
{\bfseries}
{.}
{.5em}
{#1 #2\ifx#3\relax\else~(#3)\fi}

\theoremstyle{mine}

\newtheorem{theorem}{Theorem}
\numberwithin{theorem}{section}
\newtheorem{corollary}[theorem]{Corollary}
\newtheorem{proposition}[theorem]{Proposition}
\newtheorem{lemma}[theorem]{Lemma}
\newtheorem{claim}[theorem]{Claim}
\newtheorem{definition}[theorem]{Definition}
\newtheorem{conjecture}{Conjecture} 
\numberwithin{equation}{section}

\theoremstyle{remark}
\newtheorem{remark}{Remark}

\colorlet{shadecolor}{blue!10}

\let\qed=\QED

\renewcommand{\epsilon}{\varepsilon}

\newcommand{\R}{\mathbb{R}}

\newcommand{\C}{\mathbb{C}}
\newcommand{\Z}{\mathbb{Z}}
\newcommand{\N}{\mathbb{N}}
\renewcommand{\S}{\mathbb{S}}

\newcommand{\1}{\mathbf 1}
\renewcommand{\d}{\mathbf d}

\def\T{\mathbb{T}}

\def\diam{\mathrm{diam}}

\def\var{\operatorname{Var}}

\def\P{\mathbb{P}} 
\def\E{\mathbb{E}} 
\def\md{\mid}

\def \eps {\epsilon}

\def\Bb#1#2{{\def\md{\bigm| }#1\bigl[#2\bigr]}}

\def\Eb{\Bb\E}

\def\FK#1#2#3{{\def\md{\bigm| } \P_{#1}^{\,#2}  \bigl[  #3 \bigr]}}
\def\EFK#1#2#3{{\def\md{\bigm| } \E_{#1}^{\,#2}  \bigl[  #3 \bigr]}}

\def \p {{\partial}}

\def\<#1{\langle #1\rangle}

\def\nn{\nonumber}
\def\bi{\begin{itemize}}  
\def\ei{\end{itemize}}
\def\bnum{\begin{enumerate}} 
\def\enum{\end{enumerate}}
\def\ni{\noindent}
\def\bf{\bfseries}

\def\IV{\mathrm{IV}}

\def\Villain{\mathrm{Villain}}

\title[Vortex fluctuations in $2d$ Coulomb gas and maximum of the IV-GFF]
{
Quantitative bounds on vortex fluctuations in $2d$ Coulomb gas and maximum of the integer-valued Gaussian free field
}

\author{Christophe Garban}
\author{Avelio Sepúlveda}

\address
{Université Claude Bernard Lyon 1, CNRS UMR 5208, Institut Camille Jordan, 69622 Villeurbanne, France \, and Institut Universitaire de France (IUF)}
\email{garban@math.univ-lyon1.fr}

\address{Universidad de Chile,  Departamento de Ingeniería Matemática and Centro de Modelamiento Matemático (AFB170001), UMI-CNRS 2807, Beauchef 851, Santiago, Chile.}
\email{lsepulveda@dim.uchile.cl}

\begin{document}

\maketitle

\begin{abstract}
In this paper, we study the influence of the vortices on the fluctuations of $2d$ systems such as the Coulomb gas, the Villain model or the integer-valued Gaussian free field. In the case of the $2d$ Villain model, we prove that the fluctuations induced by the vortices are at least of the same order of magnitude as the ones produced by the spin-wave. We obtain the following quantitative upper-bound on the two-point correlation in $\Z^2$ when $\beta>1$
\begin{align*}\label{}
\<{\sigma_x \sigma_y}_{\beta}^\Villain \leq C \, \left( \frac 1 {\|x-y\|_2}\right)^{\frac 1 {2\pi \beta}\left ( 1+\beta e^{-\frac{(2\pi)^2}{2} \beta}\right )}
\end{align*}
The proof is entirely non-perturbative. Furthermore it provides a new and algorithmically efficient way of sampling the $2d$ Coulomb gas.
For the $2d$ Coulomb gas, we obtain the following lower bound on its fluctuations at high inverse temperature
\begin{align*}\label{}
\EFK{\beta}{Coul}{\<{\Delta^{-1}q, g}} \geq \exp(-\pi^2 \beta + o(\beta)) \<{g,(-\Delta)^{-1}g} 
\end{align*}
This estimate coincides with the predictions based on a RG analysis from \cite{Kadanoff} and 
suggests that the Coulomb potential $\Delta^{-1}q$ at inverse temperature $\beta$ should scale like a Gaussian free field of inverse temperature of order $\exp(\pi^2 \beta)$. 

Finally, we transfer the above vortex fluctuations via a duality identity to the integer-valued GFF by showing that its maximum deviates in a quantitative way from the maximum of a usual GFF.
More precisely, we show that with high probability when $\beta>1$
\[
\max_{x\in [-n,n]^2} \Psi_n(x) \leq    \sqrt{\frac{2\beta}{\pi} \big(1 - \beta e^{- \frac{(2\pi)^2\beta} {2} } \big)} \log n \,.
\]
where $\Psi_n$ is an integer-valued GFF in the box $[-n,n]^2$ at inverse temperature $\beta^{-1}$. 
Applications to the free-energies of the Coulomb gas, the Villain model and the integer-valued GFF are also considered.
\end{abstract}

\section{Introduction}
Vortices play a fundamental role in the large scale fluctuations of statistical physics models in $2d$ such as the XY (plane rotator) model or the Villain model. 
Their statistics, especially in the case of the Villain model, are described by a celebrated statistical physics model called the (lattice)-$2d$ Coulomb gas.
Upper bounds on the fluctuations of these systems in the low temperature regime have been analyzed in the seminal work by Fröhlich and Spencer \cite{FS} and lead to the first rigorous proof  of the existence of a Berezinskii-Kosterlitz-Thouless phase transition. (See also the recent proofs \cite{aizenman2021depinning,van2022elementary} which rely on the delocalization result from \cite{lammers2022height}). As we shall explain further below (see Remark \ref{r.imposible_FS}), there is no direct way to tune the proof from \cite{FS} to provide lower bounds on fluctuations. In the case of the Villain model, lower bounds on fluctuations are equivalent to upper bounds on the two-point function $\<{\sigma_x \sigma_y}$ and the best  upper bounds known so far on the latter are given by the celebrated McBryan-Spencer estimate \cite{McBryanSpencer}. These bounds capture the fluctuations produced by the {\em Gaussian spin-wave} but do not quantify the amount of fluctuations coming from the vortices (i.e. the topological defects).

This work focuses on getting lower bounds on the fluctuations induced by the vortices in such $2d$ systems. As it has been highlighted recently in \cite{chatterjeeLowerBounds}, few techniques are available when one wants to lower-bound the fluctuations of a system for which moments are not easily under reach. A general method is developed there which leads to new quantitative lower bounds on the fluctuations of processes such as the Traveling Salesman Problem or First Passage Percolation. Closer to our present setting, fluctuations of {\em one-component} continuous Coulomb gases in $\R^2$ (with a confining potential) have been the subject of several  remarkable advances lately \cite{virag2007,makarov2011,LebleSerfaty2018,RolandPaul2016,serfaty2020}
(see Remark \ref{r.coulomb}).

In this work, we introduce a novel technique to bound from below the fluctuations of a $2d$ Coulomb gas defined on any planar lattice. Our approach is somewhat analogous to the introduction of the so-called FK-representation for Ising and Potts models. In the latter, given a spin configuration $\{\sigma_x\}_{x\in V} \in \{1,\ldots,q^{\mathrm{Potts}}\}^V$ we sample a percolation configuration $\{\omega_e\}_{e\in E}\in \{0,1\}^E$. In this paper, given a spin configuration of a Villain model $\{\sigma_x\}_{x\in V}\in (\S^1)^V$, we sample a configuration $\{m_e\}_{e\in E} \in \Z^E$ (conditionally) independently on each edge in such a way that the {\em annealed} law of $\{m_e\}_{e\in E}$ carries the fluctuations of the Coulomb gas $q:=\d m$ (where $\d$ is the discrete exterior derivative, see Section \ref{ss.discrete}). 
From a purely algorithmic point of view, this work provides an algorithm to sample the Coulomb gas which only requires an MCMC  with local updates. Due to the long-range interactions within the Coulomb gas, this is a great improvement compared to the natural (non-local) Markov chains associated to the Coulomb-gas. See Section \ref{s.algo} for the algorithmic implications.

Another advantage of our approach is that it is fully non-perturbative and yet turns out to be  quantitative enough to provide matching lower bounds in the low-temperature regime (i.e. when $\beta \to \infty$) with the predictions based on the RG flow  from the seminal paper by José, Kadanoff, Kirkpatrick, Nelson \cite{Kadanoff}. See the discussion in Subsection \ref{ss.Kadanoff} for the relationship of our work with \cite{Kadanoff}.


\subsection{Main results.}
We now state our main results and give more links to relevant works in the literature. 
To introduce the results of this section, we will work with the two-dimensional square $\Lambda_n:=[-n,n]^2\cap \Z^2$ with either free or zero boundary conditions. For more precision on this point see Section \ref{sss.boundary_condition}.  In any of these cases, we use the (negatively-definite) Laplacian  operator $\Delta$ as well as it inverse $\Delta^{-1}$ and the inner product $\langle \cdot, \cdot \rangle$. This is further explained in Section \ref{sss.calculus}. Our methods are not specific to the $\Z^2$ lattice, but for simplicity we wrote our statements in this setting, see Remark \ref{r.lattices} for a discussion about other lattices.

Most of the results of this paper will rely on the following {\em error function} $M$ which will be used to quantify 
to which extent topological defects (vortices) contribute to the macroscopic fluctuations of these systems.
\begin{align}\label{e.M}
M(\beta):= (2\pi)^2 \beta \inf_{a\in[0,1/2]} \var^{IG}(a,(2\pi)^2\beta),
\end{align}
where $\var^{IG}(a,\beta)$ is the variance of an integer-valued Gaussian random variable centred at $a$ and with inverse-temperature $\beta$ as presented in Section \ref{ss.IG}. We are interested in this function mostly when $\beta\geq1/3$, in which case we have that $M(\beta) \geq 2\beta \exp(-\frac{(2\pi)^2}2 \beta)$ (see Corollary \ref{c.M}).

\subsubsection{Fluctuations of the $2d$ Coulomb gas.} 
A two-dimensional (lattice) Coulomb gas at inverse temperature $\beta$ is a random integer-valued function $q\in \Z^F$ on the faces $F$ of $\Lambda_n$ whose probability distribution is given by
\begin{align*}
\P_{\beta}^{Coul}( q) \propto  \exp\left (-\tfrac{(2\pi)^2 \beta} 2 \langle q,(-\Delta^{-1}) q\rangle \right )\,,
\end{align*}
where boundary conditions may be chosen to be either free or Dirichlet (for a more precise description of this model see Section \ref{ss.Coulomb}). The above distribution is sometimes called the {\em Villain gas} and belongs to a wider family of Coulomb gas which are also parametrized by an {\em activity} $z$. The above model to which we will stick to corresponds to the {\em high density} ($z\equiv \infty$) regime and is naturally in correspondence with the Villain model. (See \cite{FS,FScoulomb}).   The $2d$ Coulomb gas is of  central importance in statistical physics, they have been used for example to predict critical exponents of a great family of models (\cite{nienhuis1984}) and for the analysis of the BKT transition (\cite{FS,FScoulomb}). We refer to the very useful survey on Coulomb gases \cite{brydges1999}.  
We obtain the following lower bound on the fluctuation of $q\sim \P_\beta^{Coul}$.

\begin{theorem}\label{th.Coul}
Let $q\sim\P_\beta^{Coul}$ be a Coulomb gas on the faces of $\Lambda_n$ (equipped with free or Dirichlet boundary conditions) and let $g$ be a function from the faces to $\R$. 
\begin{align}\label{e.variance_coul}
\var_{\beta}^{Coul}
\left[ \langle \Delta^{-1} q,g \rangle\right] \geq \frac{M(\beta)}{(2\pi)^2\beta} \langle g, (-\Delta)^{-1}g \rangle.
\end{align}
When the function $g$ is local (i.e. with bounded support as $\Lambda_n \nearrow \Z^2$), we obtain the following strengthened lower-bound on fluctuations which coincides, as $\beta\to \infty$ with the RG analysis from \cite{Kadanoff}
\begin{align}\label{e.variance_coul2}
\var_{\beta}^{Coul}
\left[ \langle \Delta^{-1} q,g \rangle\right] \geq e^{-\pi^2 \beta + o(\beta)} \langle g, (-\Delta)^{-1}g \rangle,
\end{align}
as long as $n\geq n_0(g)$.
\end{theorem}
\ni
The  lower bound~\eqref{e.variance_coul} on the variance of $\<{\Delta^{-1}q,g}$ extends to the following estimate on the characteristic function when $g:=\1_{f}-\1_{f'}$ for any two faces $f,f'$ in $\Lambda_n$.
\begin{theorem}\label{th.Coul2}
There exists a constant $K$ such that for all $\beta\geq \tfrac13$,  $n\geq1$ and any faces $f,f'$ in $\Lambda_n$ (again with free or $0$ b.c.s), 
\begin{align*}
\EFK{\beta}{Coul}
{e^{2i\pi (\Delta^{-1} q (f) - \Delta ^{-1} q(f'))}} \leq K\, e^{- \tfrac12 \frac{M(\beta)}{2(2\pi)^2\beta}\langle 1_{f} -1_{f'} , -\Delta^{-1}(1_{f}-1_{f'}) \rangle }.
\end{align*} 
(N.B. Note that as opposed to~\eqref{e.variance_coul}, the analogue of this estimate cannot possibly hold for all test functions $g$ as it would prevent the existence of a BKT transition). 
\end{theorem}

%


\begin{remark}\label{r.UB}
In the opposite direction, lower bounds on the characteristic function for general test functions $g$ follow  from \cite{FS}. Indeed it follows from their proof that there exists a function $\beta \mapsto \eps(\beta)$ converging to 0 as $\beta\to \infty$ s.t. for $\beta$ large enough and for any test function $g$,
\[
\EFK{\beta,\Lambda_n}{Coul}{e^{2i\pi \<{\Delta^{-1}q, g}}} \geq \exp\left (-\frac{\eps(\beta)}\beta \<{g, (-\Delta)^{-1} g}\right )\,.
\]
\end{remark}

\begin{remark}\label{}
Very precise results on the behaviour of  {\em low density} (i.e. small activity $z$) $2d$ Coulomb gas have been obtained using rigorous renormalization group methods in  works by Dimock-Hurd \cite{DimockHurd} and more recently by Falco \cite{Falco}. The difference with our present work is that our technique is non-perturbative (i.e does not rely on any renormalization group scheme) and also that it addresses the {\em high density} case corresponding to $z\equiv \infty$. (N.B. in \cite{Falco}, the focus is on the low density critical exponents near the critical temperature $T_{BKT}$. It is possible that his rigorous renormalization group methods may be adapted to also treat the high density regime $z\equiv \infty$ at high $\beta$).
\end{remark}

\begin{remark}\label{r.coulomb}
As mentioned above, our Theorem \ref{th.Coul} shares some similarities with results proved recently on the fluctuations of {\em one-component} Coulomb gases on $\R^2$. The latter is defined for a given $\beta>0$ as the following probability measure on points $\{z_1,\ldots,z_N\}\subset \R^2$
\[
\P_\beta^{Coul,\R^2}(\{z_1,\ldots,z_N\})\propto \prod_{i<j}|z_i-z_k|^\beta \exp(-N\sum_{i=1}^N\,|z_i|^2)\,.
\]
Choosing notations compatible with our setup, if $q$ denotes the empirical measure $q:=\sum_i \delta_{z_i}$, it is shown for $\beta=2$ in \cite{virag2007,makarov2011})  and for general $\beta>0$ in \cite{LebleSerfaty2018,RolandPaul2016,serfaty2020,leble2020local} 
that the centred potential $\Delta^{-1}[q-\Eb{q}]$ converges to a field which is locally absolutely continuous w.r.t the Gaussian free field at inverse-temperature $\beta$. Note that these results imply a  full CLT towards a limiting GFF while we only prove here a lower bound. A notable difference between the continuous Coulomb gas and our discrete one is the dependance of the fluctuations on the inverse temperature $\beta$. The effective inverse temperature is linear in $\beta$ for the continuous plasma and should be in $\exp(\pi^2 \beta +o(\beta))$ for the discrete plasma. See Conjecture \ref{conj}.  Even closer to our setting let us mention the work \cite{LebleSerfatyZeitouni2017} which provides large deviations estimates for the {\em two-component} Coulomb gases on $\R^2$.
\end{remark}

\subsubsection{Villain model.} The Villain model is a random function of the vertices of the graph $\Lambda_n$ taking values in the angles $[0,2\pi)$. Its measure is absolutely continuous w.r.t the Lebesgue measure on $[0,2\pi)^{\Lambda_n}$ with Radon-Nykodim derivative given by
\begin{align*}
\P_{\beta}^{Vil}(d\theta)\propto \prod_{x\sim y} \sum_{k\in \N}\exp\left (-\frac{\beta}{2}(\theta(x)-\theta(y)+2\pi k)^2\right ) d\theta.
\end{align*}
See Section \ref{ss.Vil} to see a more detailed discussion of this model.

It is well-known\footnote{This result for Villain follows from the decoupling of the Villain model into a Gaussian spin-wave times a Coulomb gas which goes back to \cite{Kadanoff}. For the plane rotator (XY) model, the analogous upper-bound is given by McBryan-Spencer's bound \cite{McBryanSpencer}.} that if $\theta$ is a Villain model in a graph $\Lambda\subseteq \Z^2$, we have that
\begin{align*}
\E^{Vil}_\beta\left [\cos(\theta(x)-\theta(y))\right ]\leq \E^{GFF}_\beta\left [e^{i(\phi(x)-\phi(y))}\right ],
\end{align*}
where $\phi$ is a two-dimensional GFF living on the same graph $\Lambda$ as $\theta$ (with same boundary conditions) and at the same inverse temperature $\beta$.

The main result of this paper concerning the Villain model is the following.
\begin{theorem}[Improved spin-wave estimate]\label{th.ISW1} Let $\theta$ be a Villain model in a graph $\Lambda_n$. There exists a constant $K$ such that uniformly in the inverse-temperature $\beta\geq 1/3$ and for all $v_1,v_2 \in \Lambda$
\begin{align*}
\E^{Vil}_{\beta}\left[\cos(\theta(v_1)-\theta(v_2))\right] &\leq K\E^{GFF}_\beta\left [e^{i(\phi(v_1)-\phi(v_2))}\right ]^{1+\frac{M(\beta)}{2}}\\
&= K\E^{GFF}_{\frac{\beta}{1+\frac{M(\beta)}{2}}}\left [e^{i(\phi(v_1)-\phi(v_2))}\right ]\,.
\end{align*}
N.B recall the definition of the error function $M(\cdot)$  in~\eqref{e.M}, see also Corollary \ref{c.M} for its behaviour at large inverse temperature $\beta$
\end{theorem}

In particular, this implies the following corollary.
\begin{corollary}\label{c.ISW2}
Consider the Villain model at inverse-temperature $\beta$ either on the graph $\Lambda_n$ or its infinite volume limit\footnote{We consider here the unique translation invariant Gibbs state, \cite{messager1978}.} on $\Z^2$. Then the following results hold uniformly in $n\in \N$ and $\beta\geq \tfrac 13$:
\begin{enumerate}
	\item Assume that $\Lambda_n$ has zero-boundary condition and take $\delta>0$. Then, there exists a constant $K=K(\delta)>0$ such that 
	\begin{align*}
	& \E^{Vil}_{\beta}\left[\cos(\theta(0)) \right]\leq Kn^{-\frac{1+\frac{M(\beta)}{2}}{4\pi \beta}}\\
	& \E^{Vil}_{\beta}\left[\cos(\theta(0)-\theta(v)) \right]\leq K \|v\|_2^{-\frac{1+\frac{M(\beta)}{2}}{2\pi \beta}}\,,
	\end{align*}
	for all  $v\in \Lambda_{\lfloor(1-\delta) n \rfloor}$. 
	\item Assume now, that $\Lambda_n$ has free-boundary condition. Then for any $v\in \partial \Lambda_n$ we have that
	\begin{align*}
	 \E^{Vil}_{\beta}\left[\cos(\theta(0)-\theta(v)) \right]\leq K\, \|v\|_2^{-\frac{3(1+\frac{M(\beta)}{2})}{4\pi \beta}}
	\end{align*}
	\item For the infinite volume limit on $\Z^2$, there exists a constant $K$ such that for all $\beta \geq \tfrac 13$ and all $v\in \Z^2$ 
	\begin{align*}
	& \E^{Vil}_{\beta,\Z^2}\left[\cos(\theta(0)-\theta(v)) \right]\leq K \|v\|_2^{-\frac{1+\frac{M(\beta)}{2}}{2\pi \beta}}\,.
	\end{align*}
\end{enumerate}
\end{corollary}

Similarly as in Remark \ref{r.UB}, in the opposite direction, one of the main results in \cite{FS} is the following lower bound on the two-point function which allowed them to provide the first rigorous evidence of the existence of the BKT transition (i.e. power-law decay of correlations at large $\beta$ versus exponential decay of correlations at small $\beta$). 
\begin{theorem}[\cite{FS}, see also the useful survey \cite{RonFS}]\label{th.FS}
There exists $\beta_0$ and a function $\eps(\beta)=o(1)$ as $\beta\to\infty$ such that if 
$\theta$ is either an XY or a Villain model on $\Z^2$ at inverse temperature $\beta\geq \beta_0$, then there is a constant $K>0$ s.t. for all $v\in \Z^2\setminus \{0\}$,  
\begin{align*}
\E^{Vil}_{\beta}\left[\cos(\theta(0)-\theta(v)) \right]\geq K^{-1} 
\|v\|_2^{-\frac{1+\eps(\beta)}{2\pi \beta}}\,.
\end{align*} 
\end{theorem}
As such, we may view our improved spin-wave estimate Theorem \ref{th.ISW1} as a quantitative lower bound on the correction exponent $\eps(\beta)$ in \cite{FS}, namely when $\beta$ is large enough, we must have by combining Theorems \ref{th.ISW1} and \ref{th.FS}
\begin{align*}\label{}
\eps(\beta) \geq \tfrac {M(\beta)}2 \geq \beta \exp(-\frac{(2\pi)^2}2 \beta)\,.
\end{align*}
In \cite{FS1983}, Fröhlich and Spencer conjectured that at low temperature, there exists an {\em effective inverse-temperature} $\beta_{eff}=\beta_{eff}(\beta)$ such that if $\theta\sim \P_{\beta,\Z^2}^{Vil}$, then the complex field $\{e^{i \theta_x}\}_{x\in \Z^2}$ fluctuates at large scales like $\exp\left(i \beta_{eff}^{-1/2} \phi_x\right)$ where $\phi$ is a Gaussian free field (with inverse temperature 1). Our statistical reconstruction analysis in \cite{GS1} was in fact partly motivated by this conjectured behaviour of the Villain model. 
This motivates the definition below for $\beta_{eff}$ which we will call $\hat \beta_{eff}$ as several different definitions of $\beta_{eff}$ could be considered. Bauerschmidt-Rodriguez-Park obtained in the recent breakthrough works \cite{bauerschmidt2022discreteNo1,bauerschmidt2022discreteNo2} the first proof of the existence of such an effective temperature in the context of the integer-valued GFF (defined in the next subsection).
\begin{definition}[Effective inverse temperature]\label{d.betaeff}
\begin{align}\label{}
\hat \beta_{eff} = \hat \beta_{eff}(\beta):= \tfrac 1 {2\pi} \limsup_{\|v\|_2 \to \infty} \frac{\log \|v\|_2} 
{\log \left(\EFK{\beta}{Vil}{\cos(\theta_0-\theta_v)}^{-1}\right)}\,.
\end{align}
\end{definition}
With this precise definition, our improved spin-wave estimate Theorem \ref{th.ISW1} thus 
implies the following bounds on the effective temperature (the lower-bound follows from \cite{FS}). See also Conjecture \ref{conj}. 
\begin{corollary}
If $\beta$ is large enough, 
\begin{align*}\label{}
\frac{\beta}{1+\eps(\beta)}\leq \hat \beta_{eff} \leq \frac{\beta}{1+M(\beta)/2} \leq \beta -\tfrac1 2 \beta^2 \exp\left (-\frac{(2\pi)^2}2 \beta\right )\,.
\end{align*}
\end{corollary}
Let us end this section on Villain by mentioning the recent related work by Dario and Wu \cite{DarioWu} which obtained 
very precise results on the truncated correlations for low temperature Villain model in $\Z^3$ (for which there is long-range order). Their work is based on a remarkable  homogenization approach which was pioneered by Naddaf-Spencer in \cite{NaddafSpencer}. Note that it seems challenging to extend such homogenization techniques to the $2d$ case with log-interactions. Also it would be interesting to obtain lower-bounds on fluctuations out of their approach.  
 

\subsubsection{ Maximum of IV-GFF.}

We now turn to an application of our improved-spin-wave estimate to the study of the maximum of the integer-valued GFF. This is the following model of functions from the vertices of $\Lambda_n$ to $\Z$ with probability distribution given by
\begin{align*}
\P^{IV}_{\beta^{-1}}(\Psi) \propto \exp\left (-\frac{1}{2\beta}\langle \Psi, (-\Delta) \Psi\rangle\right ).
\end{align*}
The integer-valued GFF can be seen as conditioning the GFF to take integer values. A further discussion on this model can be found in Section \ref{ss.IVGFF}. {\em (Note that we consider this model at inverse temperature $\beta^{-1}$ instead of $\beta$, this is due to the fact that the Villain model on $\Lambda$ at inverse temperature $\beta$ will be in correspondance with the IV-GFF on $\Lambda^*$ at inverse temperature $\beta^{-1}$)}.

In the unconditioned case, the analysis of the maximum of $2d$ GFF has attracted a lot of attention recently, see in particular the seminal works \cite{bramson2016convergence,BiskupLouidor2016}. One of the most detailed description of this maximum is the following. If $\phi_n$ is a $2d$ GFF on $\Lambda_n$ with zero boundary conditions and at inverse-temperature $\beta^{-1}$, then it is known (\cite{bramson2016convergence,BiskupLouidor2016}) that if 
\[
m_n:= \sqrt{\frac{2\beta}{\pi}}  \log n - \frac 3 8 \sqrt{\frac {2\beta} \pi}  \log\log n\,,
\]
then $\max_{x\in \Lambda_n}\{  \phi_n(x) - m_n\} $  converges in law as $n\to \infty$ towards a randomly shifted Gumbel distribution. 

If one now conditions $\phi_n$ to belong to the integers, much less is known on the behaviour of the maximum. Recently in \cite{Wirth}, Wirth relied on the techniques from \cite{FS} to show that the maximum of the IV-GFF with Dirichlet boundary conditions is a.s. larger than $c(\beta) \log n$ when its temperature $\beta$ is high enough. 
With our notations, Wirth's result can be stated as follows.
\begin{theorem}[\cite{Wirth}]\label{}
If $\beta$ is large enough, i.e. $\beta \geq \beta_0$, there exists $c(\beta)>0$ such that 
\[
\P_{\beta^{-1}}^{IV}\left[ \max_{v\in\Lambda_n} \Psi(v) \geq c(\beta) \log n\right] \to 1 \text{ as n }\to \infty \,.
\]
\end{theorem}

Our improved-spin-wave estimate allows us to provide the following non-trivial upper bound, which gives yet another indication of the role played by the vortices below $T_{KT}$.

\begin{theorem}\label{th.maxIV}
Let $\Psi$ be an IV-GFF in $\Lambda_n$ with Dirichlet boundary conditions and at inverse-temperature $\beta^{-1}$. Then for any $\beta>0$
\[
\P_{\beta^{-1}}^{IV}\left[ 
\max_{v  \in \Lambda_n} \Psi(v) \leq   \sqrt{\frac{\left( 1 - \frac{M(\beta)}{2}\right)2\beta}{\pi}}  \log n\right]
\to 1, \ \  \ \ \ \text{ as n }\to \infty. \,
\]
\end{theorem}
To prove this result 
 on the maximum of the IV-GFF, we obtain some large-deviation type estimates on $\Psi \sim \P_{\beta,\Lambda_n}^\IV$ which are reminiscent of the following works \cite{ConlonSpencer2014,BeliusWu2016} which study similar large-deviations for a class of random surfaces. The latter one for example analyses the maximum of the  {\em Ginzburg-Landau} fields. The main difference with our present work is that the models considered in \cite{ConlonSpencer2014,BeliusWu2016}  all have some convexity built-in.  In particular these work rely extensively on the Brascamp-Lieb inequality which does not hold for the present integer-valued GFF. 

Note that our statement is relevant only in the high-temperature regime for the IV-GFF where it complements the above result of Wirth. Indeed, the maximum of the integer-valued GFF in the low temperature regime can be analyzed with surprisingly high degree of precision. For example the following extremely precise asymptotics is proved in \cite{lubetzky2016}.
\begin{theorem}[Theorem 1 in \cite{lubetzky2016}]\label{}
If $\beta$ is chosen small enough (i.e. in the localized phase)\footnote{N.B. The statement is given with the convention used in our paper, i.e. the IV-GFF is always at inverse temperature $\beta^{-1}$. In \cite{lubetzky2016}, it corresponds to $\beta$ large enough.}, then there exists $n\mapsto L(n)$  with 
\begin{align*}\label{}
L(n)\sim \sqrt{(\beta/2\pi) \log n \log\log n}\,,
\end{align*}
such that
\[
\P_{\beta^{-1}}^{IV}\left( 
\max_{v  \in \Lambda_n} \Psi(v) \in\{L(n),L(n)+1\} \right) 
\longrightarrow 1 \text{ as n }\to \infty \,.
\]
\end{theorem}

\subsubsection{Estimates on free energies of Villain,Coulomb.}
The above lower bounds on vortex fluctuations imply quantitative estimates on the free energies $f^{Coul}(\beta),f^{Vil}(\beta)$ and $f^{IV}(\beta^{-1})$ for each of the models we considered. (See Corollary \ref{c.Zgff2} for their definitions).  We state below the optimized bounds which follow from Section \ref{s.Kadanoff} and which agree with the RG analysis from \cite{Kadanoff}.   
\begin{theorem}\label{th.FEs}
As $\beta\to \infty$, 
\begin{align}
\exp\left(-\pi^2 \beta(1+o(1))\right)
\leq f^{Coul}(\beta)  = 
\begin{cases}
&f^{IV}(\beta^{-1}) - f^{GFF}(\beta^{-1})  \\
&f^{Vil}(\beta) - f^{GFF}(\beta)
\end{cases}
\end{align}
\end{theorem}
\ni
We conjecture that the exponent  $\pi^2 \beta$ cannot be improved. See also Conjecture \ref{conj}.

\subsection{Idea of the proof.} We first start by coupling the Villain model to a random integer field $m$ living on the edges of the graph. The introduction of $m$ is made such that the couple $(\theta,m)$ follows a law given by a Gibbs measure, meaning that $\P^{Vil}_\beta((\theta,m))\propto \exp(-\beta \mathcal H (\theta,m))$. We call this coupling $(\theta,m)$ the {\em Villain coupling}. This random variable $m$ has many interesting properties. In particular, we show that its variance is lower bounded by that of a white noise times $(2\pi)^{-2}\beta^{-1}M(\beta)$.

The second step is to bijectively map a Villain coupling $(\theta,m)$ to a pair $(\phi,q)$, where $\phi$ is a GFF on the vertices and $q$ is an independent Coulomb gas on the faces. All of these models have the same temperature $\beta$. The mapping is so that $q$ is equal to $\d m$ the curl of $m$ (see Section \ref{ss.discrete}).
 This, together with the lower bound on the variance of $m$  implies \eqref{e.variance_coul}. The result concerning the Fourier transform of the two-point function is trickier. The proof follows the same idea as the proof of the classical central limit theorem, but at a certain point we need to see that the $\ell^2$ energy of the gradient of the Green's function is well-distributed throughout the domain.

Concerning the Villain model, we  study now the inverse of the mapping from $(\theta,m)$ to $(\phi,q)$, in particular, we concentrate in how to obtain $\theta$. We see that $\theta$ will have two independent contributions, one coming from the GFF $\phi$, which is called the {\em spin-wave} contribution. The second one comes from $q$, or to be more precise from $m$. This one is obtained from the inverse Laplacian of the divergence of $m$, i.e. $\Delta^{-1} \d^*m$. To prove that this term gives a contribution of the order of $M(\beta)$ times that of $\phi$ we need to do the same treatment as for the Fourier transform of the Coulomb gas, but this time we need to work on the vertices rather than on the faces.

Finally, the study of the integer-valued GFF is done by relating the Laplace transform of the IV-GFF living on the faces and at inverse-temperature $\beta^{-1}$ with the Fourier transform of the Villain model living on the vertices and at inverse temperature $\beta$ \footnote{This inversion of temperature from Villain/Coulomb to IV-GFF is of central importance and following the tradition of many earlier works we will always use $\beta$ for the inverse-temperature of Villain or Coulomb and $\beta^{-1}$ for the inverse temperature of IV-GFF.}. This induces a relationship between the Fourier transform of the Coulomb gas and that of the IV-GFF reminiscent of the modular invariance of the Riemann theta function. This relationship allows us to transfer the result from one model to the other. Finally, the upper bound on the maximum is obtained thanks to the fact that we control the Laplace transform of the $1$-point function of the IV-GFF.

\begin{remark}\label{r.YM}
The above idea of proof may adapt to more general settings. For example in \cite{GS3},  we apply this strategy to obtain improved spin-wave estimates for Wilson loops in $U(1)$ lattice gauge theory.
\end{remark}

\begin{remark}\label{r.lattices}
Let us point out that our techniques are based on discrete differential calculus and are as such not specific to the lattice $\Z^2$. All our statements extend to other regular lattices. For example if one would consider the triangular lattice, then statements such as Theorems \ref{th.ISW1} or \ref{th.Coul2} would look just the same except the Green function for Villain would be the Green function on the triangular grid, while the Green function $(-\Delta)^{-1}$ for the corresponding Coulomb model would be the Green function on its dual hexagonal lattice. The only place which would require some care would be the results from Section \ref{s.Kadanoff} (for example Theorem \ref{th.FEs}), where the harmonic extension of $\tfrac \pi 2\1_{x}-\tfrac \pi 2\1_{y}$ for some neighboring sites $x\sim y$ will in general give a different leading order than $\exp(-\pi^2 \beta)$. 
\end{remark}

\subsection{Organization of the paper.}
We start with a preliminaries section which gives some background on discrete differential calculus and introduces the different models studied in this paper. In Section \ref{s.coupling}, we introduce a joint coupling between a Villain model and a Coulomb gas.
We use this joint coupling in Section \ref{s.algo} to provide an efficient local sampling dynamics for the Coulomb gas. Then in Section \ref{s.transfer}, we study how to relate the IV-GFF to the Villain model and thus to the Coulomb gas. In Section \ref{s.Var}, we compare the variance of the integer-valued GFF and the unconditioned GFF at the level of variances, the bounds are then optimized to those of \cite{Kadanoff} in Section \ref{s.Kadanoff}. Section \ref{s.Fourier_Laplace} follows the line of Section \ref{s.Var} but instead of studying the variance we study the Laplace and Fourier transform of our models. Finally, in Section \ref{s.proof} we tie the final knots and we prove Theorem \ref{th.ISW1} and \ref{th.maxIV}.

The paper closes with three appendixes: Appendix \ref{a.root} analyses how the different models in this paper behave under re-rooting, Appendix \ref{a.M} gives explicit bounds on the error function $M$ while Appendix \ref{a.RP} proves the intuitive statement that the Villain field is locally nearly flat when $\beta\to\infty$.
\subsection{Acknowledgements.}
We wish to thank Roland Bauerschmidt, David Brydges and Tom Spencer for very useful discussions. We thank the referees for their very useful comments on the manuscript.
The research of C.G. is supported by the ERC grant LiKo 676999 and the research of A.S was  supported by the ERC grant LiKo 676999 and is now supported by Grant ANID AFB170001 and FONDECYT iniciación de investigación N° 11200085.

\section{Preliminaries}
\subsection{Integer-valued Gaussian random variable.}\label{ss.IG}
Integer-valued Gaussian (IG) random variables, sometimes called discrete Gaussian variables, are normal random variables conditioned to take values in $\Z$. More precisely, $X\sim \mathcal N^{IG}(a,\beta)$ if a.s. $X\in \Z$ and for any $k\in \Z$
\begin{equation}\label{e.def_IVG}
\P_{\beta,a}^{IG}\left[X=k \right]\propto e^{-\frac{\beta}{2}(k-a)^2}.
\end{equation}
In this work, we prefer to make reference to $(a,\beta)$ instead as $(\mu,\sigma^2)$ to remark that for most cases $a$ is not the mean of $X$, nor $\beta^{-1}$ is the variance of $X$. For $X$ an IG random variable of parameters $a$ and $\beta$ we define
\begin{align}
&\mu^{IG}(a,\beta):= \E^{IG}_{\beta,a}\left[X\right],\\
&\var^{IG}(a,\beta):= \var_{\beta,a}^{IG}\left[X \right]\\
&T^{IG}(a,\beta):=\E^{IG}_{\beta,a}\left[|X-\mu^{IG}(a,\beta)|^3 \right].\label{e.TIG}
\end{align}
The following proposition tells us about the smallest variance one can get for the low temperature regime.
\begin{proposition}\label{pr.Monotonicity_IG}
	Let $X$ be an integer-valued Gaussian random variable of parameter $(a,\beta)$. Then, for all $\beta>10$ and $a\in \R$
	\begin{align*}
	\var^{IG}(a,\beta)\geq \frac{1}{16}e^{-\frac{\beta (1-2a)}{2}}
	\end{align*}
\end{proposition}
The proof of this Proposition is given in Appendix \ref{a.M}. We expect that a stronger monotonicity argument should hold but our analytical argument in the appendix is not sufficient for this. In other words, we \textit{expect} that for all $a\in \R$
\begin{align}\label{e.monotonicity_variance}
&\var^{IG}(0,\beta)\leq \var^{IG}(a,\beta)\leq \var^{IG}(0.5, \beta)
\end{align}
Let us recall the definition of  the error function $\beta\mapsto M(\beta)$ which will be used throughout in this text.
\begin{align}
M(\beta):= (2\pi)^2\beta\inf_{a\in[0,1/2]}\var^{IG}(a,(2\pi)^2\beta).
\end{align}
An immediate consequence of the above proposition, which will also be shown in Appendix \ref{a.M}, is the following result on the behaviour of the error function $M(\beta)$ as $\beta \to \infty$.
\begin{corollary}\label{c.M}
We have $M(\beta) \asymp \beta \exp(-\frac{(2\pi)^2}{2} \beta)$ as $\beta\to \infty$ and we will use extensively  the following explicit lower bound: for any $\beta\geq\tfrac13$, 
\begin{align}\label{}
M(\beta) \geq 2\beta \exp\left (-\frac{(2\pi)^2}{2} \beta\right )\,.
\end{align}
{\em N.B.Assuming the above monotonicity on
$\var^{IG}(a,\beta)$ one should have
that as $\beta\to \infty$, $M(\beta) \sim 2(2\pi)^2 \beta e^{-\frac{(2\pi)^2} 2 \beta}$.}
\end{corollary}

Finally, we will need the following control (also proved in Appendix \ref{a.M}) on the ratio between the second and third moment of an IV-Gaussian.
\begin{lemma}\label{l.ratioK}
	We have that
	\begin{align}\label{e.K_beta}
	K_\beta:=
	\sup_{\hat \beta \geq \beta} \sup_{a\in \R} \frac{T^{IG}(a,\hat \beta)}{\var^{IG}(a,\hat \beta)}\in(0,\infty)
	\end{align}
\end{lemma}

\subsection{Discrete differential calculus.}\label{ss.discrete}
In this subsection, we give a short presentation of discrete differential calculus based on \cite{Roland}. For further results, see \cite{Roland,Sourav}.
  
\subsubsection{Boundary conditions.} \label{sss.boundary_condition}
Most of the result of this papers are not concern with a specific graph $G$ but only with the fact that it is two-dimensional. However, to fix ideas and to obtain results more aligned with the literature it will be simpler to only work with certain specific graphs.

\begin{definition}[Boundary conditions]We define the two types of graph we are interested in
	\begin{enumerate}
		\item \textbf{Free boundary condition:} We say that we are working with free boundary condition, if the underlying graph is $\Lambda_n^{free}:=[-n,n]^2\cap \Z^2$ embedded in the complex sphere. When we need a marked point $v_0$, unless stated otherwise, we will take $v_0=(0,0)$. When we need a marked face we typically take $f_0$ to be the one containing the point $\infty$, i.e., the only face with more than $4$ neighbours.
		\begin{figure}[h!]
			\includegraphics[width=0.3\textwidth]{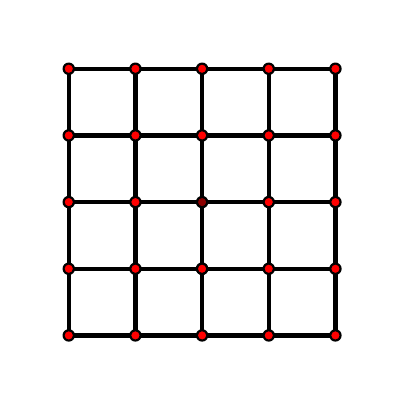}
			\includegraphics[width=0.3\textwidth]{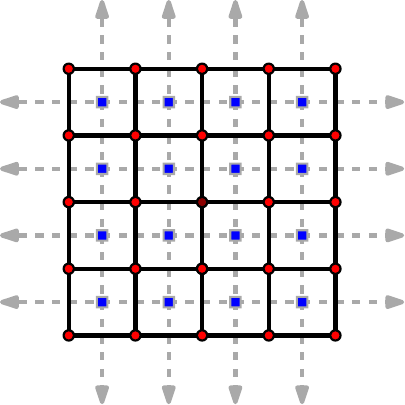}
			\caption{To the left the graph associated with free boundary condition, normally we take $v_0=(0,0)$ the dark red point. To the right one can also see the dual of the graph, in this graph $f_0=\infty$ the face connected to all the outgoing arrows.}
		\end{figure}
		\item \textbf{Zero boundary condition:} We say that we are working with zero boundary condition, if the underlying graph is the dual of a free-boundary graph\footnote{Note here that the dual of $\Lambda_n^{free}$ defined in (1) cannot be of this form for parity reasons. If one wants to see $\Lambda_n^0$ as the dual of a free-boundary graph, we would need to consider the set of vertices $([0,2n+1]^2 \cap \Z^2)- (n+\tfrac 1 2, n+\tfrac 1 2)$ instead of $\Lambda_{n+1}^{free}$. For simplicitly, we will only work with domains centred at 0.}. In other words, the graph is $\Lambda_n^{0}:=([-n,n]^2\cap \Z) \cup \{\infty \}$ where we connect all vertices in $\partial [-n,n]^2$ to the vertex $\infty$. This graph is naturally embedded in the sphere. We choose, unless specified otherwise, the marked vertex $v_0=\infty$ and the marked face $f_0$ as a face in the middle of the graph.
		\begin{figure}[h!]
			\includegraphics[width=0.3\textwidth]{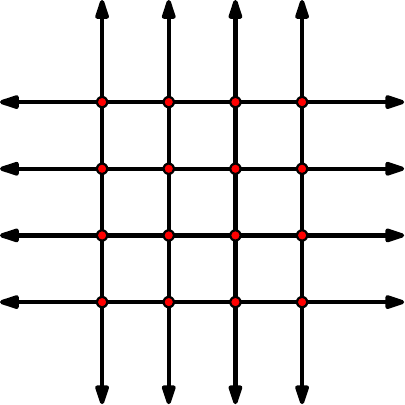}
			\includegraphics[width=0.3\textwidth]{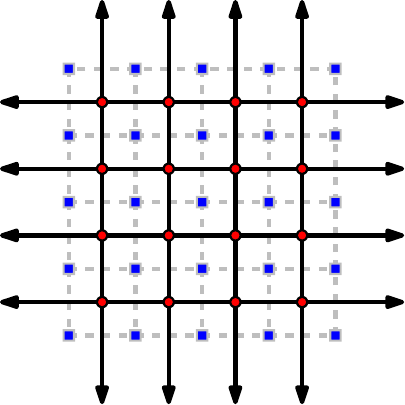}
			\caption{To the left, the graph associated with zero boundary condition. To the right one can also see the dual of the graph.}
		\end{figure} 
	\end{enumerate}
\end{definition}

\subsubsection{Basics of discrete differential calculus.}\label{sss.calculus}
Let $\Lambda_n$ be a graph defined as in the subsection before. We denote the set of vertices by $V$, the set of edges by $E$ and the sets of its faces by $F$. We call $\overrightarrow{E}$ the set of oriented edges and $\overrightarrow{F}$ the set of (cyclically-)oriented faces. Similarly to differential geometry we will denote
\begin{itemize}
	\item 0-form: functions $w:V\mapsto \R$ with $w(v_0)=0$.
	\item 1-form: functions $h:\overrightarrow E\mapsto \R $ such that for all $\overrightarrow{e}\in \overrightarrow E$, $h(\overrightarrow e)=-h(\overleftarrow{e})$. 
	\item 2-form: functions $g:\overrightarrow F\mapsto \R$ such that for all $\overrightarrow{f}\in \overrightarrow F$, $g(\overrightarrow f)=-g(\overleftarrow{f})$ and $g(f_0)=0$. 
\end{itemize}

\begin{remark}\label{}
Notice here that $0$-forms and $2$-forms are rooted (resp. in $v_0$ and $f_0$), but $1$-forms are not. These choices will make the Laplacian operator $\Delta$ (defined below) invertible on each of these  $k$-forms. 
\end{remark}

Let us, now, define an inner product.
 To do that, it is useful to fix an orientation for $E$. That is to say to consider $E$ as a subset of $\overrightarrow{E}$	where for each $e\in E$ either $\overrightarrow{e}\in \overrightarrow{E}$ or (exclusively)  $\overleftarrow{e}\in \overrightarrow{E}$.
\begin{align*}
&\langle w_1, w_2 \rangle = \sum_{v\in V} w_1(v)w_2(v),\\
&\langle h_1, h_2 \rangle = \frac{1}{2}\sum_{\overrightarrow e \in \overrightarrow{E}} h_1(\overrightarrow e)h_2(\overrightarrow{e})= \sum_{e\in E} h_1(e)h_2(e),\\
&\langle g_1,g_2\rangle = \frac{1}{2}\sum_{\overrightarrow f \in \overrightarrow{f}} g_1(\overrightarrow f)g_2(\overrightarrow{f})= \sum_{f\in f} g_1(f)g_2(f).
\end{align*}

We, now, define an operator $\d$, the discrete exterior derivative, that transforms an $n$-form to an $n+1$-form, in the following way
\begin{align*}
&\d w ((v_1,v_2))= w(v_2)-w(v_1),\\
&\d h(\overrightarrow{f})=\1_{f\neq f_0}\left( \sum_{\overrightarrow{e}\in \overrightarrow{f}} h(\overrightarrow{e})\right),\\
& \d g \equiv 0.
\end{align*}
Note that $\d$ is a linear function, and thus we define $\d^*$ as $-\d^t$. That is to say \footnote{The choice of the minus sign is because we want our Laplacian to be negative definite, as in the continuum case. We follow analyst's convention here.}
\begin{align*}
&\d^*h(v)= \1_{v\neq v_0}\left( \sum_{v'\sim v} h(\overrightarrow{vv'})\right), \\
&\d^*g(\overrightarrow{e})=\sum_{\overrightarrow{f} \owns \overrightarrow{e}} h(\overleftarrow{f}). \\
&\d^* w \equiv 0
\end{align*}

Finally, let us define the Laplacian operator
\begin{align*}
\Delta= \d \d^*+ \d^*\d,
\end{align*}
and note that it commutes with both $\d$ and $\d^*$.

The main usefulness of the operator $\d$ is given in the following well-known proposition. (We refer the reader for example to \cite{Sourav} where a slightly different setup is used).
\begin{proposition}\label{pr.basic_calculus}The following are true
	\begin{enumerate}
		\item $\d \d=0$, and thus for all  0-forms $w$ and all 2-forms $g$ we have that
		\begin{equation}
		\langle \d w, \d^*g\rangle =0.
		\end{equation}
		\item(Discrete Poincaré lemma for $\d$) For all $1,2$-forms $f$  such that $\d f=0$, we have that ther exists a $k-1$-form $n_f$ such that
		\begin{align*}
		\d n_f = f.
		\end{align*}
		Furthermore, if $f$ takes values in the integers, one can choose $n_f$ also taking values in the integers.
		\item(Discrete Poincaré lemma for $\d^*$) For all $0,1$-forms $f$  such that $\d^* f=0$, we have that there exists a $k+1$-form $n^*_f$ such that
		\begin{align*}
		\d^* n^*_f = f.
		\end{align*}
		Furthermore, if $f$ takes values in the integers, one can choose $n^*_f$ also taking values in the integers.
		
		\item The Laplacian operator is strictly negative definite (in all domains, i.e., in the vertices, faces, edges), in particular it is an invertible operator. 		
	\end{enumerate}
	
\end{proposition}

\begin{remark}\label{r.no pinned}
Note that we made a choice here to fix the values of $0$-forms, resp. $2$-forms, to $0$ in $v_0$, resp. $f_0$.
	It is also possible to study $0$ and $2$-forms that are not pinned in the marked point $v_0$ and $f_0$. In this case, one needs to properly adjust the definition of $\d$ and $\d^*$ so that it is not always $0$ in $f_0$ and $v_0$ respectively.
In this framework, one has that the statement (2) and (3) of Proposition \ref{pr.basic_calculus} are not always true. One needs to ask that $w$, resp. $f$ are such that $\sum w(v)=0$ and  $\sum g(f)=0$\footnote{For this sum to make sense, one needs a special orientation on the faces. This orientation is given on each face by the outside normal vector on the sphere.}.
	\end{remark}

\subsection{Some statistical physics models.} In this section, we will rigorously introduce the models we are going to work with, as well as proving some basic results that will be useful later.
\subsubsection{Gaussian free field.} \label{ss.GFF}

We say that a $0$-form $\phi$ is a GFF if 
\begin{align*}
\P_{\beta}^{GFF}(d\phi)\propto e^{-\frac{\beta}{2}\langle \d\phi,\d\phi \rangle} d\phi.
\end{align*}
We will sometimes also need to work with GFF that are $2$-forms. In this case, $\phi$ is a GFF if
\begin{align*}
\P_{\beta}^{GFF}(d\phi)\propto e^{-\frac{\beta}{2}\langle \d^*\phi,\d^*\phi \rangle} d\phi.
\end{align*}
In general, a GFF is a $k$-form such that
\begin{align*}
\P_{\beta}^{GFF}(d\phi)\propto e^{-\frac{\beta}{2}\langle \phi,(-\Delta)\phi \rangle} d\phi.
\end{align*}
We will call $Z^{GFF}_\beta$ its corresponding partition function.

An important property of the GFF is that it has the law of a Gaussian vector. More precisely, it can be characterised as the centred Gaussian vector with covariance
\begin{align*}
\E^{GFF}_{\beta}\left[\phi(v_1)\phi(v_2) \right]=  \frac{1}{\beta } G(v_1,v_2).
\end{align*}
Here, $G(\cdot,\cdot)$ is the Green's function of the graph-Laplacian for the associated graph with boundary condition $0$ in the marked vertex or face. Whenever we need to specify the graph $\Lambda$, we will write it on subscript, i.e., $G_\Lambda(\cdot,\cdot)$. We shall need the following asymptotics of the Green's function in $\Lambda_n$ as $n\to \infty$.
\begin{proposition}\label{pr.green}
We have the following well-known asymptotics\footnote{
Statement (1) can be found for example in Proposition 6.3.2 in \cite{LawlerLimic} or can by obtained direct computation from Stirling's formula. Statement (2) follows by the same analysis as for Theorem 4.4.4 in \cite{LawlerLimic} or by translating this estimate in terms of simple RWs. Finally, the third estimate follows by using the reflection principle (before applying the reflection principle, it is convenient to add self-loops on boundary vertices and two self-loops on corners vortices: these self-loops do not affect the operator $\Delta$ but make the correspondence with the reflected walk easier).}
for the Green's function:
	\begin{enumerate}
		\item Let us work with $0$-boundary condition, i.e., the graph $\Lambda_n^0$. Then, 
		\begin{align*}
		& G_{\Lambda_n^0}(0,0) = \frac{1}{2\pi} \log(n) + O(1),
		\end{align*}
		 here, $O(1)$ is with respect to $n$.
		\item  Let us work with free-boundary condition, i.e., the graph $\Lambda_n^{free}$. Then, for all $0<\epsilon<1$ and all $v\in \Lambda_{\lfloor(1-\epsilon) n \rfloor}$
		\begin{align*}
		&G_{\Lambda_n^{free},v_0=0}(v,v) = \frac{1}{\pi}\log(|v|+1)+O(1).
		\end{align*}
		Additionally, for all $v\in \partial \Lambda_n$ we have that 
		\begin{align*}
		G_{\Lambda_n^{free},v_0=0}(v,v) \geq \frac{3}{2\pi} \log(|v|) + O(1).
		\end{align*}
(With an equality $+O(1)$ rather than a lower bound if $v$ is at macroscopic distance from the corners of $\p\Lambda_n$).
	\end{enumerate}
\end{proposition}

The GFF is deeply related to the white noise. Here, by a white noise $W$ we understand a $1$-form with probability measure associated with
\begin{align*}
\P_{\beta}^{WN}\left(dW \right)\propto e^{-\frac{\beta}{2}\langle W,W\rangle } dW.
\end{align*}
In other words, $(W(e))_{e\in E}$ is given by independent centred Gaussian random variables with variance $\beta^{-1}$. The main relation between GFF is given in the following proposition. A closely related result can be found in \cite{Aru} and the continuum case is done in \cite{JC}.
\begin{proposition}\label{pr.GFF_WN}
	Let $W$ be a white noise at inverse-temperature $\beta$ and define
	\begin{align}
	&\phi:= \Delta^{-1}\d^* W,\\
	&\widetilde \phi:= \Delta^{-1}\d W.
	\end{align}
	Then, $\phi$ and $\widetilde \phi$ are independent, $\phi$ has the law of a GFF on the  vertices at inverse-temperature $\beta$ and $\widetilde \phi$ has the law of a GFF on the faces at inverse temperature $\beta$. 
\end{proposition}
\begin{proof}
	Let us take $w$ a $0$-form and $g$ a 2-form and note that
	\begin{align*}
	\langle\phi,w \rangle+\langle\widetilde \phi, g\rangle = -\langle W,\d \Delta^{-1}w + \d^* \Delta^{-1} g\rangle.
	\end{align*}
	Thus,
	\begin{align*}
	\E\left[e^{\langle \phi,w\rangle+\langle \widetilde \phi, g\rangle } \right] &= e^{\frac{1}{2\beta}\langle\d \Delta^{-1}w + \d^* \Delta^{-1} g,\d \Delta^{-1}w + \d^* \Delta^{-1} g\rangle }  \\
	&= e^{\frac{1}{2\beta}\langle\d \Delta^{-1}w, \d \Delta^{-1}w\rangle + \frac{1}{2\beta }\langle\d^* \Delta^{-1} g,\d^* \Delta^{-1} g\rangle }\\
	&=  e^{\frac{1}{2\beta}\langle w, (-\Delta)^{-1}w\rangle + \frac{1}{2\beta }\langle g,(-\Delta)^{-1} g\rangle },
	\end{align*}
	which is what we wanted to show.
\end{proof}
\subsubsection{Integer-valued Gaussian free field.}\label{ss.IVGFF}
The IV-GFF, $\Psi$, can be described as the GFF conditioned to take values on the integers. That is to say, $\Psi$ is either a $0$ or $2$ form taking values in the integers such that
\begin{align*}
\P_{\beta^{-1}}^{IV}(\Psi=f)\propto e^{-\frac{1}{2\beta}\langle  f,(-\Delta)f\rangle}\,,
\end{align*}
and where the boundary conditions may be chosen to be either free or Dirichlet. 

In this work, we will also study the partition function of the IV-GFF, i.e.
\begin{align*}
Z^{IV}_{\beta^{-1}}= \sum_{f} e^{-\frac{1}{2\beta}\langle\d f, \d f \rangle }\,.
\end{align*}

\subsubsection{Discrete Coulomb gas.}\label{ss.Coulomb}
A discrete Coulomb gas, $q$ is either a $0$ or a $2$ form taking values in the integers such that
\begin{align*}
\P_{\beta}^{Coul}(q=Q)\propto e^{-\frac{\beta(2\pi)^2}{2}\langle  Q,(-\Delta)^{-1}Q \rangle}\,,
\end{align*}
and where the boundary conditions may be chosen to be either free or Dirichlet. Note that the choice of boundary conditions impacts the inverse Laplacian $(-\Delta)^{-1}$. 
{\em N.B. If we consider the Coulomb gas on $0$-forms, resp. $2$-forms, then it must satisfy the constraint $q_{v_0}=0$, resp. $q_{f_0}=0$.}

The reason to add the $(2\pi)^2$ will be evident in Theorem \ref{th.coupling}, where we relate the Coulomb gas to the Villain model\footnote{If one does not want the $(2\pi)^2$ to appear, one could work with the convention that the Coulomb gas takes values in $2\pi \Z$ instead of $\Z$. We remark that we will not use this convention in our paper.}.

\begin{remark}\label{r.Coulomb_Laplacian}
	A useful characterisation of a Coulomb gas, $q$, is that it can be thought as conditioning
	\[q:=\frac{1}{2\pi}\Delta\phi,\]
	to have values on the integers. Here $\phi$ represents a (unconditioned) GFF with inverse-temperature $\beta$.
	
	An important part of the work in this paper is to compare the fluctuations of $\Delta^{-1} q$ with those of a GFF.
\end{remark}
\subsubsection{Villain model.}\label{ss.Vil}
Traditionally, a Villain model $\theta$ as introduced in \cite{Vil} is a $0$-form taking values in $[0,2\pi)$ such that its law is given by
\begin{equation}\label{e.Villain_original}
\P_{\beta}^{Vil}(d\theta)\propto \prod_{e \in E} \sum_{m\in \Z} \exp\left( -\frac{\beta}{2}  (\d\theta(e) +2\pi m)^2\right) d\theta.
\end{equation}

In this work, we change the definition of a Villain model. The main motivation is that the measure associated to $\theta$ is not given by a Gibbs measure\footnote{By this we mean that the measure cannot be written as $\propto \exp(-\beta \mathcal H(\theta))$, where the Hamiltonian $\mathcal H$ does not depend on $\beta$.}. 
Thus, we will couple a Villain model $\theta$ to a new random variable $m$ such that the pair $(\theta, m)$ has a distribution given by a Gibbs measure. We call the pair $(\theta,m)$ a \textbf{Villain coupling}. 

\begin{definition}[Villain coupling]\label{d.Villain}
We say that a random variable $(\theta,m)$ is distributed as a Villain coupling, if  $\theta$ is a $0$-form taking values in $[0,2\pi)$, $m$  is a $1$-form taking values in $\Z$ and their joint law is given by
	\begin{align}\label{e.Villain_ours}
	\P_{\beta}^{Vil}(d\theta, dm)\propto \exp\left(-\frac{\beta}{2}\langle \d\theta +2\pi m, \d\theta + 2\pi m \rangle \right) \delta_{\Z^{E}}(dm)d\theta. 
	\end{align}
\end{definition}

Some easy consequences of Definition \ref{d.Villain} are given in the following proposition.
\begin{proposition}\label{pr.Villain}Let $(\theta, m)$ be a Villain coupling, we have that
	\begin{enumerate}
		\item The marginal law of $\theta$ is that of a Villain model, i.e., \eqref{e.Villain_original}.
		\item Conditionally on $\theta$ the collection of random variables $(m(e))_{e\in E}$ are independent. Furthermore, the law of $m(e)$ conditioned on $\theta$ is that of an IV-Gaussian random variable at inverse-temperature $(2\pi)^2\beta$ and centred at $-(2\pi)^{-1}d\theta(e)$. (Recall the definition of an IV-Gaussian in \eqref{e.def_IVG}.)
	\end{enumerate}
\end{proposition}
\begin{proof}
	\begin{enumerate}
		\item It is enough to see that
		\begin{align*}
		\P^{\beta}_{Vil}(d\theta) &\propto \sum_{m\in \Z^E} \prod_{e\in E} \exp\left(-\frac{\beta}{2}( \d\theta(e) + 2\pi m(e))^2 \right)\\
		&=\prod_{e\in E}\sum_{m\in \Z} \exp\left(-\frac{\beta}{2}(\d \theta(e)+2\pi m(e))^2 \right).  
		\end{align*}
		\item This result comes directly from
		\begin{align*}
		\P^{\beta}_{Vil}(dm \mid \theta )& \propto \prod_{e \in E} \exp\left(-\frac{\beta}{2}(\d \theta(e) + 2\pi m(e))^2 \right) \delta_\Z(dm) \\
		& \propto \prod_{e \in E} \exp\left(-\frac{(2\pi)^2\beta}{2}\Big(\frac{\d \theta(e)}{2\pi} +  m(e)\Big)^2 \right) \delta_\Z(dm) . 
		\end{align*}
	\end{enumerate}
\end{proof}

\section{Coupling between Villain model and the Coulomb gas}\label{s.coupling}

In this section, we introduce a very useful coupling between $(\theta,m)$ and so ultimately between $(\theta, q)$.  Our analysis builds on a way of expressing the partition function which relies on discrete differential calculus and which goes back to the very influential paper \cite{Kadanoff}  (see also the inspiring lecture notes \cite{Roland} on this). This approach has its roots in Berezinskii's early ideas on the BKT transition (\cite{berezinskii}).

In this section, we need the following notation. Take $q$ a $2$-form, Poincaré Lemma (Proposition \ref{pr.basic_calculus} (2)) implies that there exists $n_q$ a $1$-form taking values in $\Z$ such that 
\begin{equation}\label{e.nq}
\d n_q=q.
\end{equation} From now on, for each 2-form $q$, we fix such  a (deterministic) 1-form such that $\d n_q =q$.

Now, take  $m$ a $1$-form, and define $q:=\d m$. Note  $m-n_q$ is also a $1$-form and that it satisfies $\d m-n_q =0$. As such, there exists a unique $0$-form $\psi$ such that
\begin{equation}\label{e.Psi}
\d \psi = m-n_q. 
\end{equation}

We can now formulate the following theorem.
\begin{theorem}\label{th.coupling}
	Let $(\theta,m)$ be a Villain coupling in $\Lambda_n$ (with either free or Dirichlet boundary conditions) at inverse temperature $\beta$. Let $T(\theta,m)=(\phi,q)$ be the transformation 
	\begin{align*}
	T(\theta,m):=\begin{cases}
	&q=\d m,\\
	&\phi= \theta + 2\pi \psi + 2\pi \d^*\Delta^{-1}n_q.
	\end{cases}	\end{align*}
	Then, $\phi$ is independent of $q$. Furthermore, $\phi$ is a GFF in $\Lambda_n$ and $q$ is a Coulomb gas in $\Lambda_n^*$.
	
	Furthermore, the transformation defined is invertible. In particular, take $(\phi,q)$ an independent couple where $\phi$ is a GFF in $\Lambda_n$ and $q$ is a Coulomb gas in $\Lambda_n^*$. Then,
	\begin{align*}
	&\theta = \phi-2\pi  \d^* \Delta^{-1}n_q \mod 2\pi,\\
	&m = n_q+\d \left \lfloor\frac{\phi-2\pi\d^*\Delta^{-1}n_q}{2\pi} \right \rfloor = n_q + \frac{1}{2\pi}\d(\phi-2\pi \d^*\Delta^{-1}n_q -\theta)
	\end{align*}
	is a Villain coupling in $\Lambda_n$.
\end{theorem}
\begin{remark}
The proof of this (de)coupling identity is based on the well-known duality transformation from Villain to Coulomb which goes back to \cite{Kadanoff} (see also the lecture notes \cite{Roland}).
Yet, the introduction of the coupled variable $m$ is we believe a conceptual improvement to the classical decoupling of the partition function. Indeed this allows us to provide a new construction of the Coulomb gas via $q=\d m$ which was not known before and which is in a way at the heart of our work. 
The fact that a Coulomb gas $q$ can be sampled as $\d m$, can be thought of a discrete version of the Proposition \ref{pr.GFF_WN}, where $m$ plays the role of the white noise. Note that $m$ is not a white noise, however Proposition \ref{pr.Villain} (2), shows that $m$ is close to a discrete white noise when one conditions on $\theta$. This observation will be crucial to obtain all of our estimates.
\end{remark}


\begin{proof}
	Let us first note that $T$ is a bijective transformation taking as input
	\begin{itemize}
		\item $\theta:V\mapsto [0,2\pi)$ a $0$-form;
		\item $m:E\to \Z$ a $1$-form,
	\end{itemize}
	and obtaining as output
	\begin{itemize}
		\item $\phi:V\mapsto \R$ a $0$-form.
		\item $q:F\to \R$ a $2$-form.
	\end{itemize}
	This result just follows from direct inspection of the transformations. Furthermore, define the measure
	\[\mu(d\theta,dm):=\1_{\theta\in[0,2\pi)}d\theta \delta_{\Z^E}(dm),\]
	where $d\theta$ denotes the Lebesgue measure in $\R^{V\backslash \{v_0\}}$ and $\mu$ only measures pairs $(\theta,m)$ where $\theta_{v_0}:=0$ (this follows from our definition of $0$-forms). 
	A careful look at the bijection, allows us to conclude that $T_*\mu$, the pushforward of $\mu$ by $T$, is given by the measure
	\[\nu(d\phi,dq ):= d\phi \delta_{\Z^{F\backslash\{f_0\}}}(dq),\]
	where $d\phi$ is the Lebesgue measure\footnote{Note that $d \phi$ the Lebesgue measure should not be confused below with $\d \phi$ the 1-form.} in $\R^{V\backslash \{v_0\}}$ and $\nu$ only measures pairs $(\phi,q)$ where $\phi$ is a $0$-form and $q$ a $2$-form.

	Now, let us start with $(\theta,m)$  a Villain coupling and let us compute
	\begin{align}
	\langle \d\theta +2\pi m, \d\theta +2\pi m \rangle&= \langle \d\theta + 2\pi( m-n_q + n_q), \d\theta + 2\pi( m -n_q + n_q)\rangle.\label{e.energy}	\end{align}
	Note that $m-n_q=\d\psi$ and that $n_q= (\d \d^*+\d^*\d) \Delta^{-1} n_q$, thus \eqref{e.energy} is equal to
	\begin{align*}
	&\langle \d(\theta + 2\pi(\psi + \d^* \Delta^{-1}n_q)) +2\pi\d^*\Delta^{-1}q\ , \  \d(\theta + 2\pi(\psi + \d^* \Delta^{-1}n_q)) +2\pi\d^*\Delta^{-1}q\rangle\\
	&\hspace{0.53\textwidth}=\langle \d \phi + 2\pi \d^* \Delta^{-1}q, \d \phi + 2\pi \d^* \Delta^{-1}q \rangle\\
	&\hspace{0.53\textwidth}=\langle \d \phi, \d\phi\rangle + (2\pi)^2\langle \d^* \Delta^{-1}q,  \d^* \Delta^{-1}q \rangle\\
	&\hspace{0.53\textwidth}=\langle \d \phi, \d\phi\rangle + (2\pi)^2\langle q, (-\Delta)^{-1} q \rangle.
	\end{align*}
	This implies that if $(\theta, m)$ is distributed as a Villain coupling, then $\P(d\phi, dq)$ is equal to
	\begin{align*}
	&\frac{1}{Z^{Vil}_{\beta}}\exp\left (-\frac{\beta}{2}\left( \langle \d \phi, \d\phi\rangle )+ (2\pi)^2\langle q, (-\Delta)^{-1} q \rangle \right)  \right )dT_{*}\mu \\
	&=\frac{1}{Z^{Vil}_{\beta}} \left( \exp\left (-\frac{\beta}{2}\left( \langle \d \phi, \d\phi\rangle \right) \right ) d\phi\right)\left(  \exp\left (-\frac{\beta(2\pi)^2}{2}\langle q, (-\Delta)^{-1} q \rangle \right )\delta_{\Z^{F\backslash\{f_0\}}}(dq)\right),
	\end{align*}
	which allows us to conclude.
\end{proof}

\begin{remark}\label{r.decoupling} 
	A careful look at the proof of this result gives us that
	\begin{align}
	Z^{Vil}_\beta= Z^{GFF}_\beta Z^{Coul}_\beta, \label{e.GFFxCoul}
	\end{align}
	where the GFF and the Villain model live on the vertices while the Coulomb model lives on the faces. This is the classical decoupling result for example given in Section 5.3.3 of \cite{Roland}.
\end{remark}

Theorem \ref{th.coupling} implies the following formulae.
\begin{corollary}\label{c.2-point_Vil}
	Let $(\theta,m)$ be a Villain coupling in $\Lambda_n$ with inverse-temperature $\beta$. We have that for any $v_1,v_2\in V$
	\[\E_{\beta}^{Vil}\left[\cos(\theta(v_1)-\theta(v_2) ) \right]=\E_{\beta}^{GFF}\left[e^{i (\phi(v_1)-\phi(v_2))} \right]\E_{\beta}^{Vil}\left[e^{i2\pi(\d^*\Delta^{-1}m(v_1)-\d^*\Delta^{-1}m(v_2) )}  \right].  \] 
\end{corollary}
\begin{proof}
	Note that using symmetry and Theorem \ref{th.coupling} we have that \begin{align*}
\E\left[\cos(\theta(v_1)-\theta(v_2) \right] &= \E\left[e^{i(\theta(v_1)-\theta(v_2))} \right]\\
&=
\E_{\beta}^{GFF}\left[e^{i (\phi(v_1)-\phi(v_2))} \right]\E_{\beta}^{Coul}\left[e^{i2\pi(\d^*\Delta^{-1}n_q(v_1)-\d^*\Delta^{-1}n_q(v_2) )}  \right].
	\end{align*}
	Now, we just need to check that 
	\begin{equation}
	e^{i2\pi(\d^*\Delta^{-1}n_q(v_1)-\d^*\Delta^{-1}n_q(v_2)} )=e^{i2\pi(\d^*\Delta^{-1}m(v_1)-\d^*\Delta^{-1}m(v_2) )}.
	\end{equation}
	To show this, we start by defining $\gamma$ an (oriented) edge path from $v_1$ to $v_2$ and define
	\begin{equation}\label{e.E_gamma}
	E_{\gamma}(\overrightarrow{e})=\1_{\overrightarrow{e}\in \gamma}-\1_{\overleftarrow{e}\in \gamma},
	\end{equation}
	it is then clear that $\d^*E_{\gamma}=\1_{v_1}-\1_{v_2}$. Thus, using $\d n_q=\d m$ we have that
\begin{align*}
\d^*\Delta^{-1}n_q(v_1)-\d^*\Delta^{-1}n_q(v_2)&= \langle \d^* \Delta^{-1} n_q, \d^* E_{\gamma}\rangle\\
&=-\langle \d\d^* \Delta^{-1} n_q, E_{\gamma}\rangle\\
&=\langle \d^*\d \Delta^{-1}m, E_{\gamma}\rangle - \langle n_q,E_{\gamma}\rangle\\
&=-\langle \d \d^*\Delta^{-1}m, E_{\gamma}\rangle-\langle n_q,E_{\gamma}\rangle+\langle m,E_{\gamma}\rangle\\
&=\d^*\Delta^{-1}m(v_1)-\d^*\Delta^{-1}m(v_2)-\langle n_q,E_{\gamma}\rangle+\langle m,E_{\gamma}\rangle.
	\end{align*}
	We conclude by noting that $\langle n_q,E_{\gamma}\rangle$ and $\langle m,E_{\gamma}\rangle$ are integers.
\end{proof}

\section{A local algorithm to sample the Coulomb gas}\label{s.algo}

In this section, we  consider the Coulomb gas defined on the faces of the finite box $\Lambda=\Lambda_n=[-n,n]^2 \subset \Z^2$ with free or Dirichlet Boundary conditions as defined in Subsection \ref{ss.Coulomb}.

{\em (Recall that if we impose Dirichlet boundary conditions on the $0$-forms, it will translate into the dual free boundary conditions on the $2$ forms and vice-versa. In particular the associated Villain model on $0$-form comes with the dual boundary conditions imposed on the Coulomb}).

The most natural way of sampling a {\em high density} ($z \equiv \infty$) Coulomb gas in such a  box $\Lambda\subset \Z^d$  is to run a Markov chain which at each step chooses a site $x$ uniformly at random and resamples the charge $q_x$ at site $x$\footnote{One may be worried of the lack of neutrality under such a dynamics, but recall that in our present setup the neutrality is taken care of at the root of the $2$-forms.} according to the detailed balanced conditions given by the Gibbs measure of the Coulomb gas. This if for example the Monte Carlo algorithm for sampling particles in One-Component Plasma used in \cite{brush1966monte}.

The difficulty of this approach is that the interaction $\<{q,(-\Delta)^{-1} q}$ is non-local and as such computing detailed-balanced conditions are highly time-consumming. 

Instead, Theorem \ref{th.coupling} can be used to run  a much faster local Markov chain.  Let us first explain how it works in the $2d$ case and we shall explain in \cite{GS3} how it can be extended to a sampling procedure for the Coulomb gas in dimensions $d\geq 3$. 
 
\medskip 
\ni
\textbf{A local dynamics in $2d$ to sample a Coulomb gas.}
The sampling algorithm which follows from Theorem \ref{th.coupling} consists in the following steps. 
\bnum
\item Fix $\Lambda\subset \Z^2$ a finite box,  choose free or Dirichlet boundary conditions for the Coulomb gas, and fix an inverse-temperature $\beta>0$.

\item The next step is to sample the Villain model $\theta=(\theta_i)_{i \in \Lambda}$ on the vertices of $\Lambda$ equipped with dual boundary conditions and at same inverse temperature $\beta$.
 This can be done using a local MCMC, for example by a standard Glauber dynamics or by relying on faster cluster algorithms such as in \cite{michel2015event}.
At low temperature, for the standard Gibbs dynamics, one should expect that $O(|\Lambda|^{\eta(\beta)})$ steps will be sufficient where the dynamic exponent $\eta(\beta)$ will depend on the temperature. Such a polynomial decay is not proved for the Villain or XY model but has been proved in the case of critical $2d$  Ising model in \cite{lubetzky2012critical}. 

\item Once the Villain configuration has reached equilibrium, sample the random $1$-form $m$ conditionally independently on each edge of $\Lambda$ as in Section \ref{s.coupling}:
\begin{align*}
		\P^{\beta}_{Vil}(dm \mid \theta )
		& \propto \prod_{e \in E} \exp\left(-\frac{(2\pi)^2\beta}{2}\Big(\frac{\d \theta(e)}{2\pi} +  m(e)\Big)^2 \right) \delta_\Z(dm)\,. 
		\end{align*}
		This requires $O(|\Lambda|)$ steps. 
\item Finally, one obtains the desired Coulomb gas using $q=\d m$ which also requires $O(|\Lambda|)$ final steps. 
\enum 

\begin{remark}\label{}
It turns out that the same local algorithm can be used in order to sample the Coulomb gas in any dimensions $d\geq 2$ but it requires two words of caution.
\bi
\item First, if one considers the Coulomb gas on the $d$-cells of a subgraph $\Lambda$ of $\Z^d$, one should then consider the Villain model on the $d-2$ cells. 
\item Second, the decoupling result (Theorem \ref{th.coupling}) should be generalized to a decoupling result on the $d-2$ forms and this turns out to be more subtle than when $d=2$. In fact this decoupling is a crucial part in \cite{GS3} where $U(1)$ lattice gauge theory on $\Z^4$ is analyzed. 
We will therefore discuss the generalization of this sampling algorithm for the Coulomb gas to dimensions $d\geq 3$ only in \cite{GS3} and we will see that the decoupling naturally arises at the level of $d-1$ forms instead of $d-2$ forms. We refer to \cite[Section 7]{GS3}. 
\ei
\end{remark}

\section{Transfer formula between Villain model and the IV-GFF}\label{s.transfer}
This section builds on transfer formulas listed in \cite[Theorem 3.6]{FScoulomb} (see also the review paper \cite{RonFS}). In this section, we follow the basic ideas of \cite{FScoulomb} to state a generalisation of their transfer formula but in a form which is, we believe, more transparent. 

\begin{theorem}\label{th.Fourier}
	Let $(\theta,m)$ be a Villain coupling in $\Lambda_n$ at inverse temperature $\beta$ and let $\Psi$ be an integer-valued GFF in $\Lambda_n^*$ at inverse temperature $\beta ^{-1}$. Then, for any complex $1$-form $h:E\mapsto\C$, we have that
	\begin{equation}\label{e.Fourier}
	\E_{\beta}^{Vil}\left[e^{i\langle \d \theta+2\pi m, h \rangle} \right]= \E_{\beta^{-1}}^{IV}\left[ e^{\frac{1}{\beta} \langle\d^*\Psi, h\rangle} \right] e^{-\frac{1}{2\beta}\langle h,h\rangle}.
	\end{equation} 
\end{theorem}
Let us note here that even though we are working with function taking values in the complex numbers, the inner product appearing in \eqref{e.Fourier} is not Hermitian.

Let us remark that the proof below of this transfer formula is based on the specific cases where $h\in \{-1,0,1\}$ in \cite[Theorem 3.6]{FScoulomb} and that the extension to non-integer valued functions $h$ requires the introduction of the random 1-form $m$ instead of the Coulomb gas $q$. (We also point here that this extension to non-integer $h$ will be essential in Section \ref{ss.max} where we need the Laplace transform of the IV-GFF to be defined on the full $\R^{\Lambda^*}$ in order to study its maximum).
\begin{proof}
	We start by noting that for any $\vartheta\in \R$, and any constant $\alpha\in \C$, $\beta>0$
	\begin{equation}\label{e.Fourier_series}
	\sum_{m \in \Z} e^{i\alpha(\vartheta +2\pi m )}e^{-\frac{\beta}{2}(\vartheta + 2\pi m)^2}= \frac{1}{\sqrt{2\pi \beta}}\sum_{n\in \Z} e^{-\frac{1}{2\beta}(n-\alpha)^2}e^{i n \vartheta}.
	\end{equation}
	This follows readily from the Poisson summation formula applied to  $f(t)=e^{i \alpha t -\frac {\beta} 2 t^2} \in L_1(\R,\C)$ (which has sufficient decay as $|t|\to \infty$),  or by checking directly that for any $n\in \Z$, 
the $n$-th Fourier coefficient of the function on the left is equal to 
	\begin{align*}
	c_n&=\frac{1}{2\pi}\int_{0}^{2\pi}\sum_{m \in \Z} e^{i\alpha(\vartheta +2\pi m )}e^{-\frac{\beta}{2}(\vartheta + 2\pi m)^2} e^{-in \vartheta} d\vartheta\\
	&=\frac{1}{2\pi}\int_{-\infty}^{\infty} e^{-\frac{\beta}{2}\vartheta^2} e^{-i(n-\alpha) \vartheta} d\vartheta\\
	&=\frac{1}{\sqrt{2\pi \beta}}e^{-\frac{1}{2\beta}(n-\alpha)^2}\,.
	\end{align*}	
We now use \eqref{e.Fourier_series} to see that
	\begin{align*}
\E_\beta^{Vil}\left[e^{i\langle \d \theta+ 2\pi m, h \rangle} \right]&= \frac{1}{Z^{Vil}_{\beta}}\int_{[0,2\pi]^{\Lambda_n\backslash \{v_0\}}}\sum_{m\in \Z^E}e^{i\langle \d \theta+ 2\pi m, h \rangle} e^{-\frac{\beta}{2}\langle \d\theta +2\pi m, \d \theta + 2\pi m\rangle} d\theta\\
&=\frac{1}{Z^{Vil}_{\beta}\sqrt{2\pi \beta}^{|E|} } \sum_{n \in \Z^E} e^{-\frac{1}{2\beta} \langle n-h,n-h \rangle} \int_{[0,2\pi]^{\Lambda_n\backslash \{v_0\}}}e^{i\langle n, \d \theta \rangle} d \theta.
	\end{align*}
	We now note that the integral is not $0$ only when $\d^* n$ equals $0$. When this is the case there exists a unique 2-form $\psi$ from $F$ to $\Z$ such that $\d^*\psi = n$ (thanks to Proposition \ref{pr.basic_calculus} (3)). This discussion implies that
		\begin{align*}
		\E_\beta^{Vil}\left[e^{i\langle \d \theta+2\pi m , h \rangle} \right]&=\frac{\sqrt{2\pi }^{2|\Lambda_n\backslash \{v_0\}|-|E|} }{Z^{Vil}_{\beta}\sqrt{\beta}^{|E|} } \sum_{\substack{\psi \in \Z^F\\ \psi(x_0)=0}}  e^{-\frac{1}{2\beta} \langle \d^*\psi-h,\d^*\psi-h \rangle}\\
		&= \E^{IV}_{\beta^{-1}}\left[e^{\frac{1}{\beta}\langle \d^*	\psi, h\rangle  } \right] e^{-\frac{1}{2\beta}\langle h, h \rangle }.
		\end{align*}
\end{proof}
\begin{remark}
	Note that in the last equation of the proof we identified the partition function of the integer valued GFF by studying the equality with $h=0$. This tells us that
		\begin{equation}\label{e.ZIV-ZVil}
	Z^{IV}_{\beta^{-1}}= Z^{Vil}_{\beta}\frac{\sqrt{\beta}^{|E|}} {\sqrt{2\pi }^{2|V\backslash \{v_0\}|-|E|} }= Z^{Vil}_{\beta} \frac{\sqrt{\beta}^{|V\backslash\{v_0\}|+|F\backslash \{F_0\}|} }{\sqrt{2\pi }^{|V\backslash \{v_0\}|-|F\backslash \{F_0\}|} },
	\end{equation}
	where in the last identity we used Euler's formula. Furthermore, here the integer-valued GFF is living on the faces, while the Villain model is living on the vertices.
\end{remark}

\ni
The above identity has the following two  consequences (Corollaries \ref{c.Zgff1} and \ref{c.Zgff2}). 
\begin{corollary}\label{c.Zgff1} We have that
	\begin{align}\label{e.GFFXIV}
	Z^{IV}_{\beta^{-1}}= Z^{GFF}_{\beta^{-1}} \cdot Z^{Coul}_\beta.
	\end{align}
	Here all models are living on the faces (or otherwise they are all on the vertices).
\end{corollary}
This result is  part of the {\em folklore}. It complements the classical decoupling $Z^{Vil}_\beta=Z^{GFF}_\beta\, Z^{Coul}_\beta$ we have seen earlier in Remark \ref{r.decoupling} (and for which models are {\em not} all defined on vertices). This corollary can be obtained directly using the modular invariance of the Riemann-theta function with quadratic form given by $-\Delta$. Equivalently, it can be obtained by rewriting the constraint $2\pi \sum_{n\in \Z}\delta_{2\pi n}(d\psi)$ in Fourier series as $\left(\sum_{k\in\Z} e^{i k \psi}\right) d\psi$. 
We choose to present here a proof which is specific to case of the Laplacian $\Delta$ as it fits the spirit of this paper and since it also serves as a useful consistency check.

\begin{proof}
	To be careful in this proof we will add a $*$ as a superscript to the models that are living on the faces. We combine \eqref{e.GFFxCoul} with \eqref{e.ZIV-ZVil} to obtain that
	\begin{align*}
	Z^{IV,*}_{\beta^{-1}}= Z^{Coul,*}_\beta \left( Z^{GFF}_{\beta} \frac{\sqrt{\beta}^{|\Lambda_n\backslash\{v_0\}|+|F\backslash \{F_0\}|} }{\sqrt{2\pi }^{|\Lambda_n\backslash \{v_0\}|-|F\backslash \{F_0\}|} }\right).
	\end{align*} 
	Thus, it suffices to show that the term inside the parenthesis equals $Z_{\beta^{-1}}^{GFF,*}$. Recall that
	\begin{align*}
	Z_\beta^{GFF}= \sqrt{\frac{2\pi}{\beta} }^{|\Lambda_n\backslash \{v_0\}|} \sqrt{\det(-\Delta^{-1})}.
	\end{align*}
	We recall that the determinant of the Laplacian is taken in the space of $0$ forms (which in our setup are rooted in $v_0$ and as such $\Delta$  is invertible). Thus, we have that
	\begin{align*}
	Z^{GFF}_{\beta} \frac{\sqrt{\beta}^{|\Lambda_n\backslash\{v_0\}|+|F\backslash \{F_0\}|} }{\sqrt{2\pi }^{|\Lambda_n\backslash \{v_0\}|-|F\backslash \{F_0\}|} }= \sqrt{2\pi\beta}^{|F\backslash \{f_0\}|} \sqrt{\det(-\Delta^{-1})}.
	\end{align*}
	We conclude using the fact that $\det(-\Delta^0)=\det(-\Delta^*)$. This follows from the matrix-tree theorem that tells us that this number is equal to the number of spanning trees in a given graph. And thanks to duality, we know that the number of spanning trees on the vertices is equal to the number of spanning trees on the faces. 
\end{proof}

\begin{corollary}\label{c.Zgff2}
The (infinite volume) free energy $f^{IV}(\beta)$ and $f^{Coul}(\beta)$ of the Integer-Valued GFF and the Coulomb gas are well-defined and do not depend on free versus 0 boundary conditions nor on whether the models are living on the vertices or on the faces. They are related as follows: for any $\beta>0$, 
\begin{align*}\label{}
f^{Coul}(\beta) = 
\begin{cases} 
f^{IV}(\beta^{-1}) - f^{GFF}(\beta^{-1})  \\
f^{Vil}(\beta) - f^{GFF}(\beta)
\end{cases}
\end{align*}
\end{corollary}
\begin{remark}\label{}
This corollary is also well-known: see for example \cite{lieb1972} where the classical sub-additivity arguments are adapted to the case of the long-range interactions of the Coulomb gas. Similarly for the integer-valued GFF, note that the classical trick of applying sub-additivity arguments to $Z^{IV}_{\Lambda_{2^n}}$ works well with 0 boundary conditions but already for free-boundary conditions, it is much less clear than, say, for the classical sub-additivity argument for the Ising model with {\em free} boundary conditions: this is due to the non-compactness of $\Z$. As such the above identities give us precise links between
\[
Z^{Vil}_{\beta,\Lambda} \longleftrightarrow Z^{Coul}_{\frac 1 \beta, \Lambda^*}
\longleftrightarrow Z^{IV}_{\frac 1 \beta, \Lambda^*}\,,
\]
and allow us to avoid relying on potentially tedious sub-additivity arguments. We are only left with checking the convergence of the free energy for the {\em Villain model} which follows by standard arguments as spins are compactly supported in this case.
 The same story holds for Coulomb which is both non-compact and non-local.  The new input of this paper on free energies is certainly not this corollary but rather Theorem \ref{th.FEs} below which will provide quantitative bounds on these free energies which are consistent with the low-temperature expansion found in \cite{Kadanoff}. 
\end{remark}

\begin{proof}
Let us first focus on the existence of $f^{Coul}(\beta)$. The identity~\eqref{e.GFFxCoul} tells that that for any $n\geq 1$, 
\begin{align*}
Z^{Vil}_{\beta,\Lambda_n}&= Z^{GFF}_{\beta,\Lambda_n} Z^{Coul}_{\beta,\Lambda_n^*}, 
\end{align*}
which implies 
\begin{align*}
\frac{1}{|\Lambda_n^*|} \log Z^{Coul}_{\beta,\Lambda_n^*}
&=(1+o(1)) \frac{1}{|\Lambda_n|} \left(\log  Z^{Vil}_{\beta,\Lambda_n}  - \log Z^{GFF}_{\beta,\Lambda_n}\right)\,.
\end{align*}
By the classical sub-additivity argument, $\frac{1}{|\Lambda_n|} \log  Z^{Vil}_{\beta,\Lambda_n}$ converges as $n\to \infty$ to the well-defined free-energy $f^{Vil}(\beta)$. The second term  $\frac{1}{|\Lambda_n|} \log  Z^{GFF}_{\beta,\Lambda_n} \sim \tfrac12 \log(2\pi / \beta) + \frac{1}{2|\Lambda_n|}\log(-\Delta_{\Lambda_n})$ also converges. Furthermore, these limits do not depend on the prescribed boundary conditions around $\Lambda_n$. 

Finally, for the existence of $f^{IV}(\beta)$, we proceed in the same manner, but by using instead the second identity~\eqref{e.GFFXIV}.
\end{proof}

We now state on its own a particular case of Theorem \ref{th.Fourier} because its form is simplified (no need of the coupled 1-form $m$) and because it will be often used in the rest of this text.
\begin{corollary}
	In the context of Theorem \ref{th.Fourier}, if we take an integer-valued $1$-form $h:E\mapsto\Z$ we have that
	\begin{equation}\label{e.Vil_IV_int}
	\E^{Vil}_\beta\left[e^{i\langle \d \theta, h\rangle} \right]= \E^{IV}_{\beta^{-1}}\left[ e^{\frac{1}{\beta}\langle\d^*\Psi, h\rangle } \right] e^{-\frac{1}{2\beta}\langle h,h\rangle}.
	\end{equation}	
\end{corollary}

The most important consequence of Theorem \ref{th.Fourier} is the connection it has with Theorem \ref{th.coupling}. This will be a key step to estimate the Laplace transform of the IV-GFF in Section \ref{s.proof}. 
\begin{proposition}\label{pr.Coul_GFF}Let us work in the context of Theorem \ref{th.Fourier}, then for any complex $1$-form $h:E\mapsto \C$
	\begin{equation}\label{e.Coul_GFF}
\E^{IV}_{\beta^{-1}}\left[ e^{\frac{1}{\beta}\langle\d^*\Psi, h\rangle } \right]= e^{\frac{1}{2\beta}\langle \d h,(-\Delta)^{-1}\d h\rangle}\E^{Coul}_{\beta}\left [  e^{i2\pi \langle\d^* \Delta^{-1}q,h \rangle } \right].
	\end{equation}
	In other words, for any complex 2-form $g$	\begin{equation}\label{e.Coul_GFF_2}
	\E^{IV}_{\beta^{-1}}\left[ e^{\frac{1}{\beta}\langle\Psi, g\rangle } \right]= e^{\frac{1}{2\beta}\langle  g,(-\Delta)^{-1}g\rangle}\E^{Coul}_{\beta}\left [  e^{i2\pi \langle \Delta^{-1}q,g \rangle } \right].
	\end{equation}
\end{proposition}
\begin{remark}
	Let us note that the fact that the IV-GFF lives on the faces of $V$ is only made here in order to keep the spirit of Theorem \ref{th.Fourier}, as a graph on faces in $d=2$ is isomorphic to another planar graph on vertices.
\end{remark}
\begin{remark}\label{}
Let us also stress as we mentioned below Corollary \ref{c.Zgff1} that one may have obtained the above identities without relying on the Villain model but by relying instead on the modular invariance of the Riemann theta function. See \cite{GS1} where we exploited this link between IV-GFF and the Riemann-theta function. 
\end{remark}

\begin{proof} We start by noting that Theorem \ref{th.Fourier} implies that
	\begin{align*}
 \E^{IV}_{\beta^{-1}}\left[ e^{\frac{1}{\beta}\langle\d^*\Psi, h\rangle} \right] &=e^{\frac{1}{2\beta}\langle h,h\rangle}\E^{Vil}_{\beta}\left[e^{i\langle \d \theta+2\pi m, h \rangle} \right]\\
 &=e^{\frac{1}{2\beta}\langle h,h\rangle}\E^{Vil}_{\beta}\left[e^{i\langle \d \theta+2\pi ((m-n_q)+  \d \d^* \Delta^{-1}n_q), h \rangle} e^{i2\pi\langle\d^* \Delta^{-1}q,h \rangle } \right].
	\end{align*}
We can, now, use Theorem \ref{th.coupling} to see that the last equation is equal to
\begin{align*}
&e^{\frac{1}{2\beta}\langle h,h\rangle}\E^{GFF}_{\beta}\left[e^{i\langle \d \phi, h \rangle}\right] \E^{Coul}_{\beta}\left [  e^{i2\pi\langle\d^* \Delta^{-1}q,h \rangle } \right]\\
&\hspace{0.4\textwidth}=e^{\frac{1}{2\beta}\langle \d h,\d(-\Delta^{-1})h\rangle}\E^{Coul}_{\beta}\left [  e^{i2\pi\langle\d^* \Delta^{-1}q,h \rangle } \right].
\end{align*}

The proof of \eqref{e.Coul_GFF_2} comes from Proposition \ref{pr.basic_calculus}, i.e., the fact that all 2-forms $g$ are equal to $\d h$ for some $1$-form $h$. 
\end{proof}

Equation \eqref{e.Coul_GFF} gives an important insight on the law of the integer value GFF. To do that, let us note that if $\phi$ is a GFF on the faces of $\Lambda_n$ at inverse-temperature $\beta^{-1}$, we have that for any complex $1$-form $h:E\mapsto \C$
\begin{equation}
\E^{GFF}_{\beta^{-1}}\left[e^{\frac{1}{\beta}\langle \d^* \phi, H \rangle } \right]= e^{\frac{1}{2\beta}\langle \d h, \d (-\Delta^{-1}) h\rangle}. 
\end{equation}
As a consequence of this discussion, we obtain the following classical consequence of Ginibre's inequality \cite{Ginibre}. 
\begin{proposition}\label{pr.less_fluctuation}
	Let $\phi$, resp. $\Psi$, be a GFF, resp. an IV-GFF,  in $\Lambda_n$ at inverse-temperature $\beta^{-1}$ and with any possible boundary condition. Then, for any 0-form $w:V\mapsto \R $
	\begin{equation*}
	\E^{IV}_{\beta^{-1}}\left[e^{\langle \Psi, w\rangle}\right]\leq \E_{\beta^{-1}}^{GFF}\left[e^{\langle \phi, w\rangle}   \right]. 
	\end{equation*}
	In particular,
	\begin{equation}\label{e.Ginibre_IV}
	\E^{IV}_{\beta^{-1}}[\langle \Psi, w\rangle^2 ]\leq \E^{GFF}_{\beta^{-1}}[\langle \phi, w\rangle^2 ].
	\end{equation}
\end{proposition}

In fact, we can think of \eqref{e.Coul_GFF} as a quantitative improvement of Proposition \ref{pr.less_fluctuation}, which we call a {\em spin-wave improvement}. In fact, Proposition \ref{pr.less_fluctuation} is generally proven via a derivative argument which gives no exact formula for how much fluctuation is lost by going from the GFF to the IV-GFF.

Let us, now, make another remark to help us understanding \eqref{e.Coul_GFF}. Consider the term
\begin{equation}\label{e.Coulomb_side}
\E^{Coul}_{\beta}\left [  e^{i2\pi\langle\d^* \Delta^{-1}q,h \rangle } \right].
\end{equation}
Recall that the energy term coming from Coulomb is the same one as the one for $(2\pi)^{-1} \Delta \phi$, where $\phi$ is a GFF at inverse-temperature $\beta$ (see Remark \ref{r.Coulomb_Laplacian}). Thus, if $q$ were not constrained to be in $\Z$, we would have that \eqref{e.Coulomb_side} would be equal to $
\exp(-(2\beta)^{-1}\langle \d h, \d (-\Delta^{-1})h\rangle)$, which in turn would imply that \eqref{e.Coul_GFF} needs to be equal to $1$. In particular, Theorem \ref{th.Fourier} implies that to understand the fluctuation of the IV-GFF, we need to understand how far away is the law of a Coulomb gas from that of $2\pi \Delta \phi$.

To finish this section, let us discuss what Proposition \ref{pr.Coul_GFF} tells us about the relationship between the variances of the Integer-valued GFF, the GFF and the Coulomb gas.

\begin{proposition}\label{pr.var_formula}Let us work in the context of Proposition \ref{pr.Coul_GFF} and let $g$ be any real-valued 2-form, then
	\begin{align}\label{e.variance_equality}
	\var_{\beta^{-1}}^{GFF}[\langle \phi, g\rangle ] -\var_{\beta^{-1}}^{IV}\left[\langle \Psi, g \rangle \right]=(2\pi)^2\beta^2 \var_{\beta}^{Coul}\left[\langle \Delta^{-1} q, g \rangle \right].
	\end{align}
	In particular, for any $v, v' \in V$
	\begin{align}
\label{e.var_1-point}&\var_{\beta^{-1}}^{GFF}[\phi(v) ] -\var_{\beta^{-1}}^{IV}\left[\Psi(v) \right]= (2\pi)^2\beta^2\var_{\beta}^{Coul}\left[ \Delta^{-1} q (v)\right],\\
\label{e.covar}&\E_{\beta^{-1}}^{GFF}[\phi(v)\phi(v') ] -\E_{\beta^{-1}}^{IV}\left[\Psi(v)\Psi(v') \right]=(2\pi)^2\beta^2\E_{\beta}^{Coul}\left[ (\Delta^{-1} q) (v) (\Delta^{-1} q)(v') \right].
	\end{align}
\end{proposition}

\begin{remark}
	Let us note that if our graph is just made of two points, \eqref{e.variance_equality} implies that
	\begin{equation}\label{e.variance_IG}
	1-(2\pi)^2\beta \var^{IG}(0,(2\pi)^2\beta)=\beta^{-1}\var^{IG}(0,\beta^{-1})\,.
	\end{equation}
If we had the monotonicity of the variance as described in \eqref{e.monotonicity_variance}, we would have that this term would be equal to $1-M(\beta)$. Let us also recall that this last equality is a particular instance of Jacobi theta's identity which played a key role in \cite{GS1}.	
\end{remark}
\begin{proof}
	We start by noting that since $g$ is a $2$-form, we have that there exists an $h:E\mapsto \R$ such that $\d h=g$, thus Proposition \ref{pr.Coul_GFF} implies that for any $\lambda>0$
	\begin{align*}
	\E^{IV}_{\beta^{-1}}\left[e^{\frac{\lambda}{\beta} \langle\Psi,g\rangle} \right]= \E^{GFF}_{\beta^{-1}}\left[e^{\frac{\lambda}{\beta}\langle \phi,g \rangle } \right]\E^{Coul}_{\beta}\left[e^{2 i \pi \lambda \langle \Delta^{-1}q, g \rangle} \right].
	\end{align*}
	Note that as $\langle\Psi,g\rangle$, $\langle\phi,g\rangle$ and $\langle\Delta^{-1}q,g\rangle$ all have zero average, we obtain \eqref{e.variance_equality} by a simple Taylor of order two in $\lambda$. One can obtain \eqref{e.var_1-point} by taking $g=\1_{v}-\1_{v_0}$, where $v_0$ is any point in $\partial V$. Finally, \eqref{e.covar} is obtained by taking $g=\1_{v}-\1_{v_0}$ together with \eqref{e.var_1-point}
	\end{proof}
\section{Bounds on the variance and free-energy of the IV-GFF}\label{s.Var}
In this section, we will obtain the first bounds for the variance of both the Coulomb gas and the IV-GFF. This will also imply bounds on the free-energy of the IV-GFF. 

\subsection{Variance of Coulomb Gas and the IV-GFF.}
We will now concentrate in bounding the variance of a Coulomb gas. In particular, we will prove \eqref{e.variance_coul} of Theorem \ref{th.Coul}. This can be rephrased in the following proposition. 
\begin{proposition}\label{pr.Variance_Coul}
	Let $q$ be a Coulomb gas (with either boundary condition) in $F$ at inverse-temperature $\beta$. Then, for any $2$-form $g:F\mapsto \R$
	\begin{equation}\label{e.variance_bound_Coulomb}
	\var^{Coul}_{\beta}(\langle \Delta^{-1}q,g \rangle )\geq (\inf_{a\in \R} \var^{IG}(a,(2\pi)^2\beta))\langle g, (-\Delta)^{-1}g \rangle=\frac{M(\beta)}{(2\pi)^2\beta}\langle g, (-\Delta)^{-1}g \rangle.
	\end{equation}
\end{proposition}

\begin{proof}[Proof of Proposition \ref{pr.Variance_Coul}]
	We start by using Theorem \ref{th.coupling}, to note that if $(\theta,m)$ is a Villain coupling at temperature $\beta$, then $q=\d m$ has the law of a Coulomb gas. Then,
	\begin{align}
	\var^{Coul}_{\beta}\left[ \langle \Delta^{-1}q,g \rangle\right]&= \var^{Vil}_\beta\left[ \E\left[ \langle \Delta^{-1}\d m,g\rangle\mid\theta \right]\right] + \E^{Vil}_\beta\left[\var\left[\langle \Delta^{-1}\d m,g\rangle\mid \theta \right]  \right]\\
	\nonumber&\geq \E^{Vil}_\beta\left[\var\left[\langle m,\d^*\Delta^{-1}g\rangle\mid \theta \right]  \right].
	\end{align}
	Recall from Proposition \ref{pr.Villain} that $m(e)$ are independent conditionally on $\theta$, and that the law of $m(e)$ conditionally on $\theta$ is equal to a discrete Gaussian variable centred in $-(2\pi)^{-1}\d\theta(e)$ and at inverse-temperature $(2\pi)^2\beta$. This gives us 
	\begin{align}
	\label{e.Var_Coul}\var^{Coul}_{\beta}\left[ \langle \Delta^{-1}q,g \rangle\right]&\geq \sum_{e\in E}\E^{Vil}_\beta\left[ \var^{IG}\left (-\frac{\d\theta(e)}{2\pi}, (2\pi)^2\beta\right )  (\d^*\Delta^{-1}g (e))^2  \right]\\
	\nonumber &\geq \inf_{a\in \R} \var^{IG}(a,(2\pi)^2\beta)\langle \d^* \Delta^{-1}g,\d^*\Delta^{-1} g \rangle \\
	\nonumber&= \inf_{a\in \R} \var^{IG}(a,(2\pi)^2\beta)\langle g,(-\Delta)^{-1} g \rangle.
	\end{align}
	This concludes the proof.	
\end{proof}

This proposition together with \eqref{e.variance_equality} and \eqref{e.variance_IG} allows us to prove the following corollary.
\begin{corollary}\label{c.var_IV}Let $\phi$, $\Psi$ be a GFF and IV-GFF respectively in $\Lambda_n$, then for any function $g$ and any $\beta>0$,
	\begin{align*}
	\var^{IV}_{\beta^{-1}}(\langle \Psi, g\rangle)&\leq \left (1-M(\beta)\right ) \var^{GFF}_{\beta^{-1}}(\langle\phi, g\rangle ).
	\end{align*}
\end{corollary}

\subsection{Bounds on the free-energy of the IV-GFF and the Coulomb gas.}
In this section, we obtain some first quantitative bounds on the free energies $f^{Coul}$ and $f^{IV}$ which will be improved in the next section in order to prove Theorem \ref{th.FEs}.

As we have already seen that those free energies do not depend on the boundary conditions, consider for simplicity our graph $\Lambda_n$ with free boundary conditions. We will consider in this section that the Coulomb and IV-GFF models are defined on the vertices of this graph i.e. on the $0$-forms and in such a way that they are rooted at $v_0=0$. (By duality, if one prefers one may stick with $2$-forms here). 
Recall that the partition function of the GFF at inverse temperature $\beta$ is explicit and given by
\begin{align*}
Z^{GFF}_{\beta^{-1}}=Z^{GFF}_{\beta^{-1},\Lambda_n}= \int_{\phi \text{ a 0-form}} e^{-\frac{1}{2\beta}\langle\d \phi, \d \phi \rangle  }d\phi=\sqrt{2\pi \beta}^{|\Lambda_n|-1}\sqrt{\det(-\Delta^{-1})}.
\end{align*}
{\em (N.B. recall that in our present setup forms are rooted in $v_0$  and as such $-\Delta$ is positive-definite on the space of $0$-forms.)}

The following proposition is a consequence of Corollary \ref{c.var_IV}.

\begin{proposition}\label{t.bound_free_energy}
	\begin{align*}
	f^{Coul}(\beta)  = f^{IV}(\beta^{-1}) - f^{GFF}(\beta^{-1}) 
	& \geq \tfrac 12 \int_{\beta}^{\infty} v M(v) dv \\
	& \geq \exp\left(-\frac{(2\pi)^2} 2 \beta(1+o(1))\right)\,,
	\end{align*}
where the correction term $o(1)\to 0$ as $\beta \to \infty$. 	
\end{proposition}

\begin{proof}
It is sufficient to show the following bound on the relevant partition functions in $\Lambda_n$
\begin{align}\label{e.boundPF}
	Z^{Coul}_{\beta,\Lambda_n^*}=\frac{Z^{IV}_{\beta^{-1},\Lambda_n}}{Z^{GFF}_{\beta^{-1},\Lambda_n}}\geq \exp\left( \frac{(|\Lambda_n|-1)}{2} \int_{0}^{\beta^{-1}}\frac{M(u^{-1})}{u}du\right)\,.
	\end{align}

	Note that the first equality follows directly from \eqref{e.GFFXIV}.
	
	For what follows, it will be easier to work with the IV-GFF at inverse temperature $u$ (instead of $\beta^{-1}$). The reason is the following claim.
	\begin{claim}\label{c.ratio_to_1}
		As $u\to 0$
		\begin{equation}
		\frac{Z^{IV}_{u}}{Z^{GFF}_{u}}\to 1.
		\end{equation}
	\end{claim}
Let us first see how to use the claim to prove~\eqref{e.boundPF}, and we will then show the claim. Let us first note that for any $u>0$ 
\begin{align*}
&\partial_u \log(Z^{IV}_u) = -\frac{1}{2} \E_{u}^{IV}(\langle \d\Psi, \d \Psi \rangle),\\
&\partial_u \log(Z^{GFF}_u)= -\frac{1}{2}\E_{u}^{GFF}(\langle \d\phi, \d \phi \rangle)=-\frac{1}{2u}(|\Lambda_n|-1).
\end{align*} 
Now, note that by Corollary \ref{c.var_IV} we have
\begin{align}
\label{e.mean_energy_IV}\E_{u}^{IV}(\langle \d\Psi, \d \Psi \rangle)
&= \sum_{xy\in E}\E^{IV}_u\left[ (\Psi(x)-\Psi(y))^2 \right]\\
\nonumber&\leq (1-M(u^{-1})) \sum_{xy\in E}\E^{GFF}_u\left[ (\phi(x)-\phi(y) )^2\right]\\
\nonumber&= (1-M(u^{-1}))\,\, \E^{GFF}_u(\langle \d\phi, \d \phi \rangle). 
\end{align}
This leads us to 
\begin{align*}
\partial_u \log\left (\frac{Z_u^{IV}}{Z_u^{GFF}}\right )\geq\frac{1}{2} M(u^{-1})\,\E_{u}^{GFF}(\langle \d\phi, \d \phi \rangle)= (|\Lambda_n|-1)\frac{M(u^{-1})}{2u},
\end{align*}
which implies the estimate~\eqref{e.boundPF}. Thus, we are only missing a proof of the claim.
\ni
{\em Proof of Claim \ref{c.ratio_to_1}.}
	Let us first compute
	\begin{align*}
	Z^{IV}_u-Z_u^{GFF}&= \sum_{\Psi\in \Z^{\Lambda_n\backslash \{v_0\}}} e^{-\frac{u}{2}\langle \d \Psi, \d \Psi \rangle}- \int e^{-\frac{u}{2}\langle \d \phi, \d \phi \rangle} d \phi\\
	&=\frac{\sqrt{u}^{|\Lambda_n|-1}\sum_{\Psi\in \sqrt{u}\Z^{\Lambda_n\backslash \{v_0\}}} e^{-\frac{1}{2}\langle \d \Psi, \d \Psi \rangle}-\int e^{-\frac{1}{2}\langle \d \phi, \d \phi\rangle} d\phi}{\sqrt{u}^{|\Lambda_n|-1}}\,.
	\end{align*}
Thus,
\begin{align*}
\frac{Z_u^{IV}}{Z_u^{GFF}}= 1 + \frac{\sqrt{u}^{|\Lambda_n|-1}\sum_{\Psi\in \sqrt{u}\Z^{\Lambda_n\backslash \{v_0\}}} e^{-\frac{1}{2}\langle \d \Psi, \d \Psi \rangle}-\int e^{-\frac{1}{2}\langle \d \phi, \d \phi\rangle} d\phi}{\sqrt{2\pi}^{|\Lambda_n|-1}\sqrt{\det(-\Delta^{-1})}}.
\end{align*}
We conclude by using the fact the weak-convergence
\begin{align*}
a^{|\Lambda_n|-1}\sum_{\Psi \in a \Z^{\Lambda_n\backslash \{v_0\}}} \delta_{\Psi}(dx) \to dx, \ \ \text{ as } a\to 0,  
\end{align*}
where $dx$ is the Lebesgue measure in $\R^{\Lambda_n\backslash \{v_0\}}$, and the fact that $e^{-\frac{1}{2}\langle \d\phi, \d\phi\rangle}$ converges quickly enough to zero as $\| \phi \|$ goes to $\infty$. The last statement is thanks to the Poincaré inequality, which in this context means the following
\begin{align*}
\|\phi\|^2 \leq |\Lambda_n|\phi^2(\hat v)= |\Lambda_n|\langle \d \phi, \1_{\gamma} \rangle^2 \leq |\Lambda_n|^2 \langle \d \phi, \d \phi \rangle, 
\end{align*}
where $\hat v$ is a point where $\phi^2$ attains its maximum and $\gamma$ is an edge path from $v_0$ to $\Lambda_n$.
\end{proof}

\section{Matching the lower bounds on fluctuations with the RG analysis from \cite{Kadanoff}}\label{s.Kadanoff}
In this section, we optimize the technique we developed in the previous sections to obtain lower bounds on fluctuations of Villain, IV-Gaussian free field and Coulomb gas. In the limit where $\beta \to \infty$, our lower bound matches with the vortex fluctuations obtained using RG analysis in the seminal work \cite{Kadanoff}. See the discussion on the link with \cite{Kadanoff} in Subsection \ref{ss.Kadanoff} below. 
	
	To obtain our improved bounds, let us  first define
	\begin{align}
	\tilde M (\beta,e):= (2\pi)^2 \beta \E^{Vil}_\beta \left[\var^{IG}\left (-\frac{\d\theta(e)}{2\pi}, (2\pi)^2\beta\right ) \right].
	\end{align}
	Now, we note that \eqref{e.Var_Coul} may be improved without any effort to
	\begin{align}
	\nonumber\var^{Coul}_{\beta}\left[ \langle \Delta^{-1}q,g \rangle\right]&\geq \sum_{e\in E}\E^{Vil}_\beta\left[ \var^{IG}\left (-\frac{\d\theta(e)}{2\pi}, (2\pi)^2\beta\right )  (\d^*\Delta^{-1}g (e))^2  \right]\\
	\label{e.bound_with_tildeM}&\geq \frac{1}{(2\pi)^2 \beta}\langle \d^* \Delta^{-1}g\tilde M (\beta,\cdot ),\d^*\Delta^{-1} g \rangle.
	\end{align}
	
	\subsection{A better bound for $\tilde M(\beta,e)$.}
	The objective of this section is to show the following proposition. 
\begin{proposition}\label{pr.tilde_M}
For any $\delta>0$, there exists $\beta_0=\beta_0(\delta)<\infty$ such that the following holds: for any $\beta\geq \beta_0$,  there exists $L_0=L_0(\beta,\delta)\in\N$ such that  for any $n\geq L_0$, if one considers the Villain coupling $(\theta,m)$ on $\Lambda_n$ with either free or Dirichlet boundary conditions, then for all edge $e\in \Lambda_{n-L_0}$

		\begin{align*}
		\tilde M(\beta,e)\geq 2 \beta e^{-(\pi+\delta)^2 \beta }.
		\end{align*}
		
		In simpler terms, for any edge $e$ of $\Z^2$ one has that
		\begin{align*}
		\liminf_{\beta\nearrow\infty} \liminf_{\Lambda_n\nearrow \Z^2} \frac{1}{\beta}\log (\tilde M(\beta,e))\geq -\pi^2
		\end{align*}
	\end{proposition}
%

	To prove the proposition, we first fix a graph $\Lambda_R$ with zero boundary condition. We define the function $\hat f=\hat f_R$ as the harmonic extension of the function taking values $0$ at $0$ and $1$ at $1$. To clarify, $\hat f$ does not necessarily take value $0$ at $\infty$. Note that $\hat f$ is a $0$-form pinned at the point $0$.
	
	In fact, $\hat f$ has the following explicit formula  for any $z\in \Lambda_R$:
	\begin{align*}
	\hat f(z)= \frac{G_{\Lambda_{R}, v_0=0}(1,z)}{G_{\Lambda_{R}, v_0=0}(1,1)}.
	\end{align*} 
	From now on, when we talk about $\hat f$ we always fix $v_0=0$.
	
	Let us note that
	\begin{align}\label{e.GR}
	\langle \d \hat f, \d \hat f \rangle =\frac{1}{(G_{\Lambda_R}(1,1))^2} \langle G_{\Lambda_R}(1,\cdot) , -\Delta G_{\Lambda_R} (1,\cdot) \rangle= \frac{1}{G_{\Lambda_R}(1,1)}.
	\end{align}
	
	The reason we introduce $\hat f$ is given by the following statement which is an analog of Girsanov's change of measure applied to the Villain model.
	\begin{lemma}[\bf Villain-Girsanov]\label{l.girsanov}
		Let $n,R\in \N$ and $(\theta,m)$ be a Villain coupling in  $\Lambda_n$. We have the following for any  $a\in (0,\tfrac12)$, $0<\epsilon<4^{-1} \wedge 2\pi(1-2a)$ and edge $e=\overrightarrow{xy} \in \Lambda_{n-R}$
		\begin{align*}
		&\P^{Vil}_\beta(\d \theta(e)\in (2\pi a,2\pi(1 -a))\\
		&\hspace{0.075\textwidth}\geq e^{-\frac{\beta}{2}(2\pi a+ 2 \epsilon R)^2 \langle \d \hat f, \d \hat f \rangle  -R^2e^{- \beta}}\P_\beta^{Vil}(\d \theta(e') \in(-\epsilon,\epsilon), \text{for all } e' \in \Lambda_R + x )
		\end{align*}
	\end{lemma}
	\begin{proof}
		 Let $e=\overrightarrow{xy} \in \Lambda_{n-R}$, we want to define $\hat f^e$ a $0$-form in $\Lambda_n$ such that $\d \hat f^e(e)=1$ and $\langle \d \hat f, \d \hat f \rangle=\langle \d \hat f^e, \d \hat f^e \rangle$. To do this, we define
		\begin{align}
		\tilde f ^e ( v)=\begin{cases}
		\hat f(v-x) & \text{if } v \in \Lambda_R + x,\\
		\hat f(\infty) & \text{else}.
		\end{cases}
		\end{align}
		We are almost there as  $\d \tilde f^e(e)=1$ and
		$\langle \d \hat f, \d \hat f \rangle=\langle \d \tilde f^e, \d \tilde f^e \rangle$. However, $\tilde f^e$ is not a $0$-form, so we define $\hat f^e(v) = \tilde f^e(v)-\tilde f^e(v_0)$.
		
		Now, let us note that
		\begin{align}
		\nonumber&\P^{Vil}_\beta(\d \theta(e)\in (2\pi a,2\pi(1 -a))\\
		\nonumber&=\frac{1}{Z_\beta^{Vil}}\sum_{m\in \Z^{E}} \int_{[0,2\pi]^{\Lambda_n\backslash \{v_0\}}} \1_{\d \theta(e)\in (2\pi a,2\pi(1 -a))} e^{-\frac{\beta}{2}\langle \d \theta + 2\pi m,\d \theta + 2\pi m \rangle }d\theta\\
		\label{e.bounding_dtheta_1}&\geq \frac{1}{Z_\beta^{Vil}}\sum_{m\in \Z^{E}} \int_{[0,2\pi]^{\Lambda_n\backslash \{v_0\}}} \1_{\d (\theta-(2\pi a+\epsilon)\hat f^e)(e)\in (-\epsilon,\epsilon)} e^{-\frac{\beta}{2}\langle \d \theta + 2\pi m,\d \theta + 2\pi m \rangle }d\theta.
		\end{align}
		By doing the change of variable $\theta'\mapsto (\theta -(2\pi a+\epsilon) \hat f^e) \mod 2\pi$, and using the periodicity of the function $u\mapsto\sum_{m\in \Z} e^{-(u+2\pi m)^2}$, we have that \eqref{e.bounding_dtheta_1} is lower bounded by
		\begin{align}
		\nonumber&\frac{1}{Z_\beta^{Vil}}\sum_{m\in \Z^{E}} \int_{[0,2\pi]^{\Lambda_n\backslash \{v_0\}}} \1_{\d \theta'(e)\in (-\epsilon,\epsilon)} e^{-\frac{\beta}{2}\langle \d \theta' + (2\pi a+\epsilon) \d  \hat f + 2\pi m,\d \theta + (2\pi a+\epsilon) \d  \hat f + 2\pi m \rangle }d\theta'\\
		\nonumber&\geq e^{-\frac{\beta}{2}(2\pi a+\epsilon)^2 \langle \d \hat f, \d \hat f\rangle } \E^{Vil}_\beta\left[e^{-(2\pi a +\eps) \beta\langle \d \hat f^e, \d\theta + 2\pi m \rangle } \1_{\d\theta (e)\in (-\epsilon,\epsilon)} \right]\\
		\label{e.bounding_dtheta_2}&\geq e^{-\frac{\beta}{2}(2\pi a+\epsilon)^2 \langle \d \hat f, \d \hat f\rangle } \E^{Vil}_\beta\left[e^{-(2\pi a+\eps) \beta\langle \d \hat f^e, \d\theta + 2\pi m \rangle } \1_{\d\theta (e')\in (-\epsilon,\epsilon), m(e')=0 \text{ for all } e' \in x+\Lambda_{R}} \right]
		\end{align}
		Using Cauchy-Schwarz and the fact that $\d \hat f^e$ is supported on $x + \Lambda_R$ we have that
		\begin{align*}
		|\langle \d \hat f^e, 2\pi\d \theta + m\rangle | & \leq \sqrt{\langle\d \hat f ^e, \d \hat f^e \rangle \langle (\d \theta + 2\pi m) \1_{x+ \Lambda_R}, (\d \theta + 2\pi m) \1_{x+ \Lambda_R}\rangle}\\
		&\leq  \epsilon 2 R\sqrt{\langle \d \hat f^e, \d \hat f^e\rangle}
		\end{align*}
		on the event where $\d\theta (e')\in (-\epsilon,\epsilon), m(e')=0$  for all $e' \in x+\Lambda_{R} $. Thus, we have that \eqref{e.bounding_dtheta_2} is lower bounded by
		\begin{align*}
		&e^{-\frac{\beta}{2}(2\pi a+\epsilon)^2 \langle \d \hat f, \d \hat f\rangle - (2\pi a+\eps) \beta  \epsilon 2 R \sqrt{\langle \d \hat f^e, \d \hat f^e \rangle } } \E^{Vil}_\beta\left[ \1_{\d\theta (e')\in (-\epsilon,\epsilon), m(e')=0 \text{ for all } e' \in x+\Lambda_{R}} \right]\\
		&\geq e^{-\frac{\beta}{2}(2\pi a+\epsilon)^2 \langle \d \hat f, \d \hat f\rangle - (2\pi a +\eps) \beta \epsilon 2 R \sqrt{\langle \d \hat f^e, \d \hat f^e \rangle } } (1- e^{-\frac{(2\pi)^2\beta}{2}(1-4\epsilon)})^{4R^2} \times\\
		&\hspace{0.5\textwidth}\P^{Vil}_\beta( \d\theta (e')\in (-\epsilon,\epsilon) \text{ for all } e' \in x+\Lambda_{R}),
		\end{align*}
		where we used that uniformly on $\d \theta (e)\in (-\epsilon,\epsilon)$ the probability that $m\neq 0$ is less or equal than $\exp(-(2\pi)^2(1-4\epsilon)) $. We conclude by using that $\langle \d \hat f, \d \hat f\rangle\geq 1 $, $\log(1+x)\geq (1-\delta)x$ for small enough $x$ and that $4(2\pi)^2(1-4\epsilon)>2$.
	\end{proof}
	
	We now need to estimate $\langle \hat f, \hat f \rangle=\<{\hat f_R, \hat f_R}$.
	
	\begin{lemma} \label{l.energy_f}
		We have that as $R\nearrow \infty$,
		\begin{align}
		\langle \hat f, \hat f \rangle \to 2.
		\end{align}
	\end{lemma}
	\begin{proof}
		As observed in~\eqref{e.GR}, we only we need to compute $G_{\Lambda_R,v_0=0}(1,1)$. However, let us note that as $R\nearrow \infty$, $G_{\Lambda_R,v_0=0}(1,1)\to G_{\Z^2,v_0=0}(1,1)=1/2$ by equation (2.2) of \cite{KW}\footnote{The result is better represented in Figure 2.1 of \cite{KW}. Note that what we need in our case is twice the value given in Figure 2.1 of \cite{KW}, as we need $G(0,0)+G(1,1)-2G(0,1)$.}, which can also be found in Section 12 equation (3) of \cite{Spi}.
	\end{proof}

We are now ready for the proof of Proposition \ref{pr.tilde_M}. The proof below will rely on  Proposition \ref{pr.SmallGr} in the appendix which shows that in a $R$-window $x+\Lambda_R$, gradients can be made arbitrarily small when the temperature is low enough. 

\begin{proof}[Proof of Proposition \ref{pr.tilde_M}]
		We start by using Proposition \ref{pr.Monotonicity_IG} to see that for any $a\in(0,1/2)$
		\begin{align}
		\nonumber\tilde M(\beta,e)&= (2\pi)^2\beta \E^{Vil}_\beta\left[ \var^{IG}\left (\frac{d\theta}{2\pi},(2\pi)^2\beta \right )\right]\\
		&\geq \frac{(2\pi)^2 \beta}{16} e^{-\frac{(2\pi)^2}{2}(1-2a)} \P^{Vil}_\beta\left ( \d \theta(e)\in(2\pi a, 2\pi (1-a) \right ).\label{e.M(beta,e)}
		\end{align}
		We now use Lemmas \ref{l.girsanov} and \ref{l.energy_f} together with Proposition \ref{pr.SmallGr} to see that \eqref{e.M(beta,e)} is lower bounded by
		\begin{align}\label{e.last_tildeM}
		\frac{(2\pi)^2 \beta}{16} e^{-\frac{(2\pi)^2\beta}{2}(1-2a)-\frac{\beta}{2}(2\pi a+\epsilon)^2( 1+ 2\epsilon R) (2+o_R(1))  -R^2e^{- \beta}} (1- \epsilon).
		\end{align} 
		We conclude by taking 
$a=1/2-\frac 1 {10} \delta$, $R=R(\delta)$ sufficiently large so that $\<{\d \hat f, \d \hat f} \geq 2 - \frac 1 {10} \delta$,  $\eps=\eps(\delta)$ sufficiently small so that, say,  $\eps R \leq \delta^2$ and finally $\beta_0$ sufficiently large so that Proposition \ref{pr.SmallGr} applies (with the inputs $\eps(\delta)$ and $R(\delta)$) and also so that $e^{-\beta_0} R^2 \leq \delta^2$.

	\end{proof}
	\begin{remark}
		To understand the last part of the proof it may be easier to forget all the $\delta$, $\epsilon$ and $R$, to approximate \eqref{e.last_tildeM} by
		\begin{align*}
		\frac{(2\pi)^2\beta}{16} e^{-\frac{(2\pi)^2\beta}{2}(1-2a)- \beta  (2\pi)^2a^2}. 
		\end{align*} 
		The element in the exponential is minimised when $a=1/2$. For that value the exponent is equal to $\pi^2$.
	\end{remark}

	\subsection{Sharp lower bounds on the fluctuations at low temperature.}
	
	An important consequence of the lower bound obtained for $\tilde M(\beta,e)$ is the following.
	\begin{proposition}\label{pr.improved_tildeM}
		One has the following correction to the temperature term for the Coulomb gas on the faces with free or zero boundary conditions: for any given $2$-form $g$ with finite support, 
		\begin{align}\label{e.asymp_Coul}
	\liminf_{\beta\nearrow \infty}\liminf_{\Lambda_n\nearrow \Z^2}\frac{1}{\beta}(\log(\E^{Coul}_\beta \left[ \langle \Delta^{-1} q,g\rangle \right]) - \log \langle g, (-\Delta)^{-1} g \rangle )\geq - \pi^2.
		\end{align}
		This implies a similar result for the IV-GFF with free or $0$-boundary condition:
		\begin{align}\label{e.asymp_GFF}
\liminf_{\beta\nearrow \infty}\liminf_{\Lambda_n\nearrow \Z^2}\frac{1}{\beta}\left (\log\left (\E^{GFF}_{\beta^{-1}} \left[ \langle \phi ,g\rangle\right ] - \E^{IV}_{\beta^{-1}} \left[ \langle \phi ,g\rangle \right]\right ) - \log \langle g, (-\Delta)^{-1} g \rangle \right )\geq - \pi^2.
		\end{align}
	\end{proposition}
\begin{proof}
Let us note that the equality between \eqref{e.asymp_Coul} and \eqref{e.asymp_GFF} follows from \eqref{e.variance_equality}. We now concentrate in proving \eqref{e.asymp_Coul}. We start from \eqref{e.bound_with_tildeM} and fix $g$ a $2$-form
\begin{align*}
\frac{1}{(2\pi)^2\beta}\langle \d^* \Delta^{-1} g \tilde M(\beta, \cdot ) , \d^* \Delta^{-1} g \rangle&\geq \frac{1}{(2\pi)^2\beta }\langle \d^*\Delta^{-1} g  \tilde M(\beta, \cdot ) \1_{\Lambda_{n-L}} , \d^* \Delta^{-1} g \rangle,
\end{align*}
for any $L\in \N$. We now fix $\delta>0$ and take $\beta\geq \beta_0$ such that Proposition \ref{pr.tilde_M} is satisfied. We take also $L_0=L_0(\beta)$ in that proposition. Thus, we have that for all $n\geq L_0$
\begin{align*}
\var^{Coul}_\beta \langle \Delta^{-1} q, g \rangle \geq  \frac{2 e^{-(\pi +\delta)^2\beta}}{(2\pi)^2} \langle \d^* \Delta^{-1} g \1_{\Lambda_{n-L_0}}, \d^*\Delta^{-1}g \rangle.
\end{align*}
Thus we just need to show by using the fact that the support of $g$ is bounded, that one has 
\begin{align}\label{e.condition_g}
\langle \d^* \Delta^{-1} g \1_{\Lambda_{n-L_0}}, \d^*\Delta^{-1}g \rangle \geq \frac{1}{2}\langle \d^* \Delta^{-1} g , \d^*\Delta^{-1}g \rangle,
\end{align}
for all $L_0\in \N$ and for all $n$ sufficiently large.

Now let us use the following bound on the gradient of the Green function which holds uniformly in all  $u,v \in \Lambda_n$ and for either free or 0 boundary conditions,
\begin{align*}
\d^* \Delta^{-1}\1_u(v)  = \d^*G(u,v) \leq  \frac{C}{\|u-v\|}.
\end{align*}
{\em (N.B this bound is well-known).}
And thus, if $M:=\diam(\mathrm{supp}(g))$ then when $n\gg L_0$ we have that, 
\begin{align*}
\langle \d^* \Delta^{-1} g \1_{\Lambda_n\backslash\Lambda_{n-L_0-M}}, \d^*\Delta^{-1}g \rangle & \leq \left(\frac{C}{n-L_0-M} \right)^2 nL_0  \, \|g\|_\infty^2\,.
\end{align*}
which allows us to conclude as $\liminf_{n\to \infty} \langle \d^* \Delta^{-1} g , \d^*\Delta^{-1}g \rangle >0$ if $g\neq 0$.
\end{proof}
\begin{remark}
	This strengthened result on fluctuations which  hold for all finite support 2-forms $g$ is a strong hint that we should be able to change the function $M$ in all of our results for the function $e^{-(\pi+\delta)^2\beta}$. However up to this point we cannot prove this. There are two mains reasons for this.
\bi
\item[a)] The first one is that we still do not know how to control $\tilde M(\beta,e)$ close to the boundary. 
\item[b)] The second reason is deeper, and will be clear by looking at Section \ref{s.Fourier_Laplace}. Indeed in order to obtained similar strengthened bounds on the two-point function of Villain for example, we would need to control the Fourier/Laplace transform 
	\[\E^{Vil}_\beta[e^{\lambda \left \langle g \var^{IG}\left (\frac{\d \theta}{2\pi},(2\pi)^2\beta\right ),g \right \rangle }]\,.\] 
	In this case we cannot replace $\var^{IG}(\frac{\d \theta}{2\pi},(2\pi)^2\beta)$ with $\E^{Vil}_\beta\left[\var^{IG}(\frac{\d \theta}{2\pi},(2\pi)^2\beta) \right]$ without relying on some additional strong ergodicity properties.
\ei
\end{remark}

Despite the lack of control on $\tilde M(\beta,e)$ close to the boundary, we can still conclude for the behaviour of the free energies of our models at low temperature.

\subsection{Strengthened bounds on the free energies $f^{Coul}$ and $f^{IV}$.}
In this section, we  prove Theorem \ref{th.FEs}. By Corollary \ref{c.Zgff2}, it will be sufficient to show that 
as $\beta\to \infty$,
\begin{align*}
f^{Coul}(\beta)  = f^{IV}(\beta^{-1}) - f^{GFF}(\beta^{-1}) 
& \geq \exp\left(-\pi^2 \beta(1+o(1))\right)\,.
\end{align*}
\ni
{\em Proof of Theorem \ref{th.FEs}.}
Following the proof of Proposition \ref{t.bound_free_energy}, we need to improve \eqref{e.mean_energy_IV}
by replacing $M(u^{-1})$ into $\tilde M(u^{-1}, e)$ at least for edges $e$ not to close to the boundary. Indeed, exactly as in the proof of Proposition \ref{pr.improved_tildeM}, for any $\delta>0$ and if $\beta=\beta(\delta)$ is large enough, then for any temperature $u\leq \beta^{-1}$, we find a radius $R(u)<\infty$, s.t. for any edge $e=\{x,y\}$ at distance at least $R(u)$ from the boundary, then if we view $\Psi(x)-\Psi(y)$ as $\<{\Psi,g=\1_x-\1_y}$, in the proof of Proposition \ref{pr.improved_tildeM}, we find that 
\begin{align*}\label{}
\E^{IV}_u\left[ (\Psi(x)-\Psi(y))^2 \right] \leq (1 - 2 u^{-1} \exp(-(\pi+\delta)^2 u^{-1})) \E^{GFF}_u\left[ (\Psi(x)-\Psi(y))^2 \right]\,.
\end{align*}
This implies that 
\begin{align*}
& \sum_{e=\{x,y\}\in \Lambda_{n-R(u)}}\E^{IV}_u\left[ (\Psi(x)-\Psi(y))^2 \right] \\
&\leq (1-2 u^{-1} \exp(-(\pi+\delta)^2 u^{-1})) \sum_{e=\{x,y\}\in \Lambda_{n-R(u)}}\E^{GFF}_u\left[ (\phi(x)-\phi(y) )^2\right]\\
&\leq  (1-2 u^{-1} \exp(-(\pi+\delta)^2 u^{-1}))\,\, \E^{GFF}_u(\langle \d\phi, \d \phi \rangle) \\
& = (1-2 u^{-1} \exp(-(\pi+\delta)^2 u^{-1})) \frac{1}{u}(|\Lambda_n|-1)\,.
\end{align*}
Now, recall that  
\begin{align}\label{}
\E_{u}^{IV}(\langle \d\Psi, \d \Psi \rangle)
&= \sum_{xy\in E}\E^{IV}_u\left[ (\Psi(x)-\Psi(y))^2 \right]\,.
\end{align}
We are thus left with showing that as $n\to\infty$, 
\begin{align*}\label{}
\sum_{e=\{x,y\}\in \Lambda_n \setminus \Lambda_{n-R(u)}}\E^{IV}_u\left[ (\Psi(x)-\Psi(y))^2 \right] = o(|\Lambda_n|)\,.
\end{align*}
This follows from \eqref{e.Ginibre_IV} which implies that for any $x,y\in \Lambda$,
\[\E^{IV}_u\left[ (\Psi(x)-\Psi(y))^2 \right] \leq \E^{GFF}_u\left[ (\phi(x)-\phi(y))^2 \right].\]
 As the sum of the latter on  edges in $\Lambda_n \setminus \Lambda_{n-R(u)}$ is bounded by $O(1) nR(u)$, this concludes our proof.
\qed

\subsection{Discussion on the correspondance with the RG analysis in \cite{Kadanoff}.}\label{ss.Kadanoff}

In this very influential work, the authors performed a detailed renormalization group analysis of the present Villain model in order to analyse its properties and critical exponents at low temperatures. They focus in particular on the corrections which need to be added to the {\em spin-wave theory} (which in the setting of this paper correspond to the fluctuations coming from the GFF). In fact, as it was mentioned above, the decoupling of Villain into a spin-wave times a Coulomb (at least on the level of the partition function) goes back to their work. They provide two different analysis of the low-temperature regime:
\bnum
\item First (in Section III of \cite{Kadanoff}), they run a {\em Migdal recursion scheme} (which is similar to Kadanoff's classical way of implementing RG). By applying successive iterations of Migdal transformation, which correspond here to successive decimations of the $\Z^2$ lattice,  they end-up in (3.16) with the following prediction on the shift from $T$ to $T_{eff} = T_{eff}(T)$, the effective temperature of Villain model at temperature $T$:
\[
\frac{d T_{eff}} {d\ell} \sim \exp\left (-\frac {2 \pi^2}{3 T_{eff}}\right )\,.
\]
Translated to our present setup, it means Migdal's recursion method lead to the prediction that 
\[
\beta_{eff} - \beta \sim \exp\left (-\frac {2 \pi^2}{3} \beta\right )\,,
\]
which is not compatible with our results in this section. (Note that it is compatible with the non-optimized results of Section \ref{s.Var}). 

\item Fortunately, their second approach solves the above contradiction. Indeed  in Section IV of \cite{Kadanoff}, they turn the Villain model basically into a Sine-Gordon model which has now two parameters: an activity $z$ and an inverse-temperature $\beta$. In the plane $(\beta,z)$, they write down the RG equations which describe how the system flows under rescaling and averaging out the small scales fluctuations. 
The corresponding dynamical system is given by Kosterlitz equations (\cite{kosterlitz1974}). They conclude in (4.36) that, when working in the full graph $\Z^2$
\[
\beta_{eff}^1-\beta^{-1} \sim e^{-\pi^2 \beta}\,,
\]
which is exactly the lower bound we obtained in the present section for the variances of Coulomb and IV-GFF fields as well as for the shift of the free energies $f^{Coul}$, $f^{Vil}$ and $f^{IV}$. The way they arrive at $e^{-\pi^2 \beta}$ is by a computation similar to ours: they estimate in (4.15a) and (4.15b) what should be the effective activity $z_0$ before running the RG group. That effective activity is computed by somehow forcing a spin-wave to create a vortex which corresponds intuitively to our Villain-Girsanov Lemma \ref{l.girsanov}. The difference with our approach is that we do not need to view this computation as an effective activity for Sine-Gordon and we do not run any RG flow here. 
\enum
Interestingly, they write the following comment at the end of the section III about Midgal's recursion. {\em We do not yet know if we should believe the detailed predictions of the Migdal theory. (...) we shall suggest in section IV that the approximation does not treat vortices properly along the apparent fixed line.}
  
As such our Proposition \ref{pr.improved_tildeM} and Theorem \ref{th.FEs} may serve as a rigorous justification that Migdal recursion does not lead to the correct prediction for Villain model at low temperature. We end this Section with the following conjecture which is supported by our rigorous lower-bound as well as by the RG analysis from \cite{Kadanoff}. 

\begin{conjecture}\label{conj}
Recall the definition of $\hat \beta_{eff}$ in Definition \ref{d.betaeff}. 
We conjecture that as $\beta\to \infty$, the vortices contribute to the fluctuations in such a way that 
\begin{align*}\label{}
\lim_{\beta\to \infty}\frac{\log(\beta -\hat \beta_{eff})}{\beta} = -\pi^2\,.
\end{align*}
Another related way of detecting the impact of vortices is through the free energies. The same reasoning leads us to conjecture that our lower bound in Theorem \ref{th.FEs} on $f^{Coul}(\beta)$ is asymptotically sharp, i.e. \begin{align*}
\lim_{\beta\to \infty} \frac{\log f^{Coul}(\beta)}{\beta} = -\pi^2 =
\begin{cases}
\lim_{\beta\to\infty} \frac{\log\left(f^{IV}(\beta^{-1}) - f^{GFF}(\beta^{-1})\right)}{\beta}  \\
\lim_{\beta\to\infty} \frac{\log\left(f^{Vil}(\beta) - f^{GFF}(\beta)\right)}{\beta}\,.
\end{cases}
\end{align*}
\end{conjecture}

\section{Bounds on the Fourier and Laplace transform of Coulomb gases}\label{s.Fourier_Laplace}

\subsection{Fourier transform.}
We now want to upgrade the bounds on the fluctuations from Section \ref{s.Var}. We want not only to obtain bounds on the variance for the Coulomb gas and the IV-GFF, but also on their Fourier and Laplace transform. There are many reasons for this: it will allow us to control the $2$-point function of the Villain model and to control the large deviation of the value on a point of the IV-GFF.

\begin{lemma}\label{l.Fourier_m}
	Take $(\theta,m)$ a Villain coupling at inverse temperature $\beta$ and a 1-form $h:E\mapsto \R$. For $b>0$ define
	\begin{equation}
	h^{b}= h \1_{|h|<b}\,.
	\end{equation}
For all $\beta>0$ there exists $0<\tilde b(\beta)<1$ such that for all $b< \tilde b(\beta)$,
	
\begin{equation*}
	\E^{Vil}_{\beta}\left[e^{i\langle m,h \rangle} \right]\leq e^{-\frac{(1-bK_\beta)}{2} \inf_{a\in[0,1/2)}\var^{IG}(a,(2\pi)^2\beta) \langle h^b, h^b \rangle},
	\end{equation*}
	where $K_\beta$ is defined in \eqref{e.K_beta}. Furthermore, we can choose $ \beta \mapsto \tilde b(\beta)$ to be increasing.
\end{lemma}
\begin{remark}\label{}
In fact, this proposition cannot be extended to not take in to account $b$. We discuss why in Remark \ref{r.imposible_FS}.
\end{remark}
\begin{proof}
	We start as in the proof of Proposition \ref{pr.Variance_Coul}, by separating the characteristic function by conditioning on $\theta$
	\begin{align}
	\E^{Vil}_{\beta}\left[ e^{i\langle m,h\rangle }\right]&=\E^{Vil}_{\beta} \left[ e^{i\langle \E\left[ m\mid \theta \right],h\rangle }\E\left[e^{i\langle m-\E\left[m\mid \theta \right] ,h\rangle } \mid \theta \right] \right]\nn\\
	&=\E^{Vil}_{\beta} \left[ e^{i\langle \E\left[ m\mid \theta \right],h\rangle }\prod_{e\in E}\E\left[e^{i (m(e)-\E\left[m(e)\mid \theta \right]) h(e)} \mid \theta \right] \right]\,.\label{e.m(e)theta}
	\end{align}

We now use the following straightforward claim
\begin{claim}\label{}
If  $X$ is a centred random variable in $\R$ with finite third moment, then for any $h\in \R$, 
\begin{align*}
|\Eb{e^{i h X}}| & \leq  |1-\frac {h^2} 2 \Eb{X^2}| + \frac{|h|^3} 3 \Eb{|X|^3}\,.
\end{align*}
\end{claim}
\begin{proof}[Proof of the claim]
Indeed for all $x\in \R$,
	\begin{align*}
	&|\sin(x)-x|\leq \frac{|x|^3}{3!}\\
	& \left |\cos(x)-1+\frac{x^2}{2}\right |\leq \frac{|x|^3}{3!}.
	\end{align*}
This implies that 
\begin{align*}\label{}
|\Eb{e^{i h X}}|& \leq |\Eb{1- h^2\frac{X^2}2 + i X}| 
+ |\Eb{\cos(hX)-(1-\frac{h^2X^2} 2) + i\big( \sin(hX) - hX \big)}| \\
& \leq |1- \frac {h^2} 2 \Eb{X^2}| + 2 \frac{|h|^3} {3!} \Eb{|X|^3}\,.
\end{align*}	
\end{proof}

Given $\theta$, one can now use the claim  on each $e\in E$ with the centred random variable $X=X(e):=m(e)-\Eb{m(e)\md \theta}$. Plugged into~\eqref{e.m(e)theta} this implies that $|\E_{\beta}^{Vil}\left[e^{i\langle m, h\rangle} \right]|$ is smaller than or equal to (recall the definition of $T^{IG}(a,\beta)$ in~\eqref{e.TIG})
	\begin{align*}
	 &  \E_{\beta}^{Vil}\left[ \prod_{e\in E}\left |\E\left[e^{i (m(e)-\E\left[m(e)\mid \theta \right]) h(e)} \mid \theta \right]\right|  \right]\\
	&\leq \E_{\beta}^{Vil}\left[\prod_{e\in E}\left( 1-\frac{1}{2}(h^b(e))^2\var^{IG}\left (\frac{\d\theta(e)}{2\pi},(2\pi)^2\beta\right ) + \frac{|h^b(e)|^3}{3}T^{IG}\left (\frac{\d\theta(e)}{2\pi},(2\pi)^2\beta\right )\right )  \right] 
	\end{align*}
	as long as for any $e\in E$,	
	\begin{equation}\label{e.first_a_small}
	\frac{1}{2}(h^b(e))^2\var^{IG}\left (\frac{\d\theta(e)}{2\pi},(2\pi)^2\beta\right )\leq 1\,.
	\end{equation}
Note that here we have bounded by one any multiple coming from an edge where $h(e)\geq b$. This part is the first place where we need $b$ to be chosen small enough. To see that the dependence is only in $\beta$, note that for each edge $e$
\begin{align*}\label{}
\frac{1}{2}(h^b(e))^2\var^{IG}\left (\frac{\d\theta(e)}{2\pi},(2\pi)^2\beta\right ) \leq \frac 1 2 b^2 \sup_{a \in \R} \sup_{\hat \beta \geq \beta} \var^{IG}(a,(2\pi)^2\hat \beta).
\end{align*}
Note that term $\sup_{a\in \R}\sup_{\tilde \beta \geq \beta} \var^{IG}(a,(2\pi)^2\hat \beta)$ is finite. To prove this, one only needs to work with $a\in (0,1/2)$ and note that $\var^{IG}(a,\beta)$ is upper bounded by $\E^{IG}_{a,\beta}[X^2]$ where $X\sim \mathcal N^{IG}(a,\beta)$.

Let us now obtain a lower bound on the following quantity for each edge $e$
\begin{align}\nonumber
& \frac{1}{2}(h^b(e))^2\var^{IG}\left (\frac{\d\theta(e)}{2\pi},(2\pi)^2\beta\right ) -\frac{|h^b(e)|^3}{3}T^{IG}\left (\frac{\d\theta(e)}{2\pi},(2\pi)^2\beta\right ) \\
&\nonumber\geq \var^{IG}\left (\frac{\d\theta(e)}{2\pi},(2\pi)^2\beta\right ) \left( \frac 1 2 (h^b(e))^2 - b \frac{(h^b(e))^2}{3} \frac 
{T^{IG}\left (\frac{\d\theta(e)}{2\pi},(2\pi)^2\beta\right ) }
{\var^{IG}\left (\frac{\d\theta(e)}{2\pi},(2\pi)^2\beta\right )} \right) \\ 
& \nonumber\geq \var^{IG}\left (\frac{\d\theta(e)}{2\pi},(2\pi)^2\beta\right ) \left( \frac 1 2 (h^b(e))^2 - b \frac{(h^b(e))^2}{2} K_\beta \right) 
\\
& \label{e.lower_bound_Fourier}\geq  \frac 1 2 (h^b(e))^2 \inf_{a\in\R}\var^{IG}(a,(2\pi)^2\beta) \left( 1  - b  K_\beta \right) \,,
\end{align}
as long as $b< (\sup_{\tilde \beta\geq \beta} K_{\tilde \beta})^{-1}$. Note this is the second and last place where we need $b$ to be chosen small enough (only as a function of $\beta$, i.e $b<b(\beta)$). 

	We now use that $\log(1+x)\leq x$ as long as $x>-1$ together with \eqref{e.lower_bound_Fourier} to upper bound $|\E^{Vil}_{\beta}\left[e^{i\langle m, h\rangle} \right]|$ by	
	\begin{align*}
	&\E^{Vil}_\beta\left[\prod_{e\in E}e^{ -\frac{1}{2}(h^b(e))^2\var^{IG}\left (\frac{\d\theta(e)}{2\pi},(2\pi)^2\beta\right )+\frac{|h^b(e)|^3}{3!}T^{IG}(\d\theta(e),(2\pi)^2\beta^{-1})}  \right]\\
	&\hspace{0.4\textwidth}\leq e^{ -\frac{1}{2} (1-b\, K_\beta)\inf_{a\in\R}\var^{IG}(a,(2\pi)^2\beta)\langle h^b,h^b \rangle}.
	\end{align*}
	A careful study of the bounds necessary for $b$ shows that we can choose $\tilde b (\beta)$ to be increasing on $\beta$.
\end{proof}

\begin{proposition}\label{pr.Fourier_Coul}
	Let us fix $\beta, \epsilon, \bar \lambda>0$. Then, there exists a deterministic constant $K:=K(\beta,\epsilon,\bar \lambda)>0$  such that for all $|\lambda|\leq \bar \lambda$ the following is true uniformly on the size of the graph $\Lambda$
	\begin{enumerate}
		\item For all $f_1, f_2\in F$,		\begin{equation}\label{e.Fourier_Coul}
		\E_{\beta}^{Coul}\left[e^{i\lambda (\Delta^{-1}q(f_1)-\Delta^{-1}q(f_2))}\right ] \leq Ke^{-\frac{\lambda^2(1-\epsilon)}{2}G_{\Lambda^*,f_1}(f_2,f_2) \inf_{a\in \R}\var^{IG}(a,(2\pi)^2\beta)}.
		\end{equation}
		\item For all $v_1, v_2\in V$, \begin{equation}\label{e.d*m}
		\E_{\beta}^{Vil}\left[e^{i\lambda (\Delta^{-1}\d^*m(v_1)-\Delta^{-1}\d^*m(v_2) )}  \right]\leq Ke^{-\frac{\lambda^2(1-\epsilon)}{2}G_{\Lambda,v_1}(v_2,v_2)\inf_{a\in \R}\var^{IG}(a,(2\pi)^2\beta)}.
		\end{equation}
	\end{enumerate}
Furthermore, we can ask that $\beta\mapsto K(\beta,\epsilon,\bar \lambda)$ is decreasing in $\beta$.
\end{proposition}
\begin{proof}
We first start by simplifying both equations by fixing the marked faced $f_0$ to be $f_0:=f_1$ and  the marked vertex $v_0$ to be $v_0:=v_1$ (recall that $0$ and $2$-forms take value $0$ in the marked vertex and face respectively). We rely here on the re-rooting Proposition \ref{pr.rootCoul}. Thus, we only need to study the Fourier transform of $\Delta^{-1}q(f_2)$ and $\Delta^{-1}\d^*m(e_2)$.
	
Let us start with item $(1)$. Recall that $q=\d m$, thus
		\begin{equation}
		 \Delta^{-1} q(f_2) = \langle m, -\d^*\Delta^{-1}\1_{f_2} \rangle. 
		\end{equation}
		We now study the function $\widehat E:=-\d^*\Delta^{-1}\1_{f_2}$. Thanks to Lemma \ref{l.Fourier_m} and the fact that
		\begin{equation}
		\langle \widehat E, \widehat E\rangle= G_{\Lambda^*,f_1}(f_2,f_2)
		\end{equation}
		to conclude we need to show that for all $b>0$ we can find $K$ such that for all $f_1$ and $f_2$
		\begin{equation}
		\langle \widehat E^b, \widehat E^b\rangle\geq \langle \widehat E, \widehat E\rangle - K.
		\end{equation}
		This result can be proven using the following well known claim.
		\begin{claim} \label{C.gradient_Green}Uniformly on the size of the graph $\Lambda_n$ and $f_1, f_2\neq \infty$
		\begin{enumerate}
			\item $|\widehat E(e)|\to 0$ as $\|e-f_1\|\to \infty$ and $\|e-f_2\|\to \infty$.
			\item $\widehat E$ is bounded.
		\end{enumerate}
	Furthermore, in the case where $\Lambda_n$ is the graph with free-boundary condition and $f_1=\infty$, we have that uniformly on $n$ for all $f_2\in F$
\begin{enumerate}
	\item[(a)'] $|\widehat E(e)|\to 0$ as $\|e-f_2\|\to \infty$.
	\item[(b)']$\widehat E$ is bounded.
\end{enumerate}

		\end{claim}
	\begin{remark}
		Here by $\| e- f\|$ we mean the distance as a set in $\C$ between the edge $e$ and the face $f$. 
	\end{remark}
Let us say a few words on the claim. In the case of free-boundary conditions, it follows from the reflection principle together with the precise estimates on the Green function on $\Z^2$ from Theorem 4.4.4 in \cite{LawlerLimic}. (See also the footnote in Proposition \ref{pr.green}). Another useful reference for estimates on the gradient of the Green function in planar domains is Lemma B.4 from \cite{smirnov2010}.

For item $(2)$, the proof is just analogue to the case before, the only difference is that one needs to work with vertices instead than with faces.
\end{proof}

\subsection{Laplace transform.}
As in the last subsection which focused on Fourier transforms, the key step to obtain bounds on the Laplace transform of a Coulomb gas is the following lemma.
\begin{lemma}\label{l.Laplace_m}
	Take $(\theta,m)$ a Villain coupling at inverse temperature $\beta$ and a $1$-form $h:E\mapsto \R$. For $b>0$ define
	\begin{equation}
	h^{b}= h \1_{|h|<b}
	\end{equation}
	
	We have that for all $(\beta,\bar \beta)$ such that $\beta>\bar \beta>0$, there exists $0<\tilde b(\bar \beta)<1$ such that for all $b<\tilde b(\bar \beta)$
	\begin{equation*}
	\E^{Vil}_{\beta}\left[e^{\langle m,h \rangle} \right]\geq e^{\frac{1}{2}(1-b\hat K_{\bar \beta}) (\inf_{a\in \R}\var^{IG}(a,(2\pi)^2\beta)) \langle h^b, h^b \rangle},
	\end{equation*}
	where $\hat K_{\bar \beta}$ is a function only depending on $\bar \beta$.
\end{lemma}
\begin{proof}
We start in the same way as in the proof of Lemma \ref{l.Fourier_m}, by taking conditional expectation with respect to $\theta$
\begin{align*}
\E^{Vil}_{\beta}\left[ e^{\langle m,h\rangle }\right]&=\E^{Vil}_{\beta} \left[ e^{\langle \E\left[ m\mid \theta \right],h\rangle }\E\left[e^{\langle m-\E\left[m\mid \theta \right] ,h\rangle } \mid \theta \right] \right]\\
&=\E^{Vil}_{\beta} \left[ e^{\langle \E\left[ m\mid \theta \right],h\rangle }\prod_{e\in E}\E\left[e^{ (m(e)-\E\left[m(e)\mid \theta \right]) h(e)} \mid \theta \right] \right]
\end{align*}
We will now use that for all $x\in \R$ and $|y|<1/2$, we have that
\begin{align*}
&e^{x}\geq 1+x+ \frac{x^2}{2!}+\frac{x^3}{3!},\\
&\log(1+y)\geq y -8 y^2
\end{align*}
Furthermore, define $E^b\subseteq E$ as the set of all edges where $|h(e)|<b$. And let us use (2) of Proposition \ref{pr.Villain} to see that for $b<K_{\bar \beta}^{-1}\wedge 1$ (recall the definition of $K_\beta$ from~\eqref{e.K_beta}). 
\begin{align*}
&\prod_{e\in E^b}\E\left[e^{ (m(e)-\E\left[m(e)\mid \theta \right]) h(e)} \mid \theta \right]\\
&\hspace{0.1\textwidth}\geq \prod_{e\in E^b}\left (1+ \frac{1}{2}\var^{IG}\left (\frac{\d\theta(e)}{2\pi},(2\pi)^2\beta\right )h^2(e) + \frac{1}{3!}T\left (\frac{\d \theta(e)}{2\pi}, (2\pi)^2 \beta\right ) h(e)^3 \right )\\
&\hspace{0.1\textwidth}\geq \prod_{e\in E^b}\left (1+ \frac{h^2(e) }{2}\var^{IG}\left (\frac{\d\theta(e)}{2\pi},(2\pi)^2\beta\right )\left ( 1- bK_{\bar \beta}\right )\right )\\
&\hspace{0.1\textwidth}\geq \exp\left( \frac{1}{2}\inf_{a\in \R}\var^{IG}(a,(2\pi)^2\beta)\sum_{e\in E^b} h^2(e)(1-b\tilde K_{\bar \beta})\right), 
\end{align*}
for
\begin{equation*}
\tilde K_{\bar \beta}:=K_{\bar \beta} + 16 \sup_{a\in \R} \sup_{\hat \beta \geq \bar \beta}\var^{IG}(a,(2\pi)^2 \hat \beta).
\end{equation*}
Then, we have that
\begin{align*}
\E^{Vil}_{\beta}\left[ e^{\langle m,h\rangle }\right]\geq  e^{\frac{1}{2}(1-b\tilde K_{\bar \beta})\inf_{a\in \R} \var^{IG}(a,(2\pi)^2\beta) \langle h^b, h^b \rangle} \E^{Vil}_\beta\left[e^{\langle m,h-h^b\rangle + \langle \E\left[m\mid \theta \right], h^b \rangle  } \right]. 
\end{align*}
We conclude using Jensen's inequality and the fact that by symmetry. 
\begin{align*}
\E^{Vil}_\beta\left[\langle m,h-h^b\rangle + \langle \E\left[m\mid \theta \right], h^b \rangle \right ]=0.
\end{align*}
\end{proof}
Now, the exact same proof of Proposition \ref{pr.Fourier_Coul} gives the following result.
\begin{proposition}\label{pr.Laplace_Coul}
	Let us fix $\beta, \epsilon, \bar \lambda>0$. Then, there exists a constant $K>0$  such that for all $|\lambda|\leq \bar \lambda$ the following is true uniformly on the size of the graph $\Lambda$
	\begin{enumerate}
		\item For all $f_1, f_2\in F$,		\begin{equation}\label{e.Laplace_Coul}
		\E_{\beta}^{Coul}\left[e^{\lambda (\Delta^{-1}q(f_1)-\Delta^{-1}q(f_2))}\right ] \geq Ke^{\frac{\lambda^2(1-\eps)}{2}G_{\Lambda^*,f_1}(f_2,f_2) (\inf_{a\in \R}\var^{IG}(a,(2\pi)^2\beta))}.
		\end{equation}
		\item For all $v_1, v_2\in V$, \begin{equation}\label{e.Laplace_d*m}
		\E_{\beta}^{Vil}\left[e^{\lambda (\Delta^{-1}\d^*m(v_1)-\Delta^{-1}\d^*m(v_2) )}  \right]\geq Ke^{\frac{\lambda^2(1-\eps)}{2}G_{\Lambda,v_1}(v_2,v_2)(\inf_{a\in \R}\var^{IG}(a,(2\pi)^2\beta))}.
		\end{equation}
	\end{enumerate}
\end{proposition}

\section{Proof of the main Theorems}\label{s.proof} 
\subsection{Proof of the improved-spin wave estimate.}

We now prove Theorem \ref{th.ISW1}, by rewriting it in the following way.
\begin{proposition}Let us consider $\theta$ to be a Villain model in $\Lambda$ with inverse temperature $\beta$. Then for any $M^-(\beta)<M(\beta)$ there exists a constant $K:=K(\frac{M^-}{M(\beta)})> 0$, such that
	\begin{align*}
	\E^{Vil}_{\beta}\left[\cos(\theta(v_1)-\theta(v_2)) \right]\leq K e^{-\frac{1}{2\beta}(1+M^-(\beta))G_{\Lambda,v_1} (v_2,v_2)}  
	\end{align*}
	for all $v_1,v_2\in \Lambda$.
\end{proposition}
\begin{proof}
	We first note that by Corollary \ref{c.2-point_Vil}
	\begin{align*}
	\E^{Vil}_{\beta}\left[\cos(\theta(v_1)-\theta(v_2)) \right]&\leq \E_{\beta}^{GFF}\left[e^{i (\phi(v_1)-\phi(v_2))} \right]\E_{\beta}^{Coul}\left[e^{i2\pi(\d^*\Delta^{-1}m(v_1)-\d^*\Delta^{-1}m(v_2) )}  \right]\\
	&= e^{-\frac{1}{2\beta}(G_{\Lambda,v_1}(v_2,v_2))}\E_{\beta}^{Coul}\left[e^{i2\pi(\d^*\Delta^{-1}m(v_1)-\d^*\Delta^{-1}m(v_2) )}  \right]\\
	&\leq  K e^{-\frac{1}{2\beta}(1+M(\beta)(1-\epsilon))G_{\Lambda,v_1} (v_2,v_2)},
	\end{align*}
	where in the last line we used the estimate~\eqref{e.d*m} from Proposition \ref{pr.Fourier_Coul}.
\end{proof}

Let us finally remark that Corollary \ref{c.ISW2} follows just from Theorem \ref{th.ISW1} together with Proposition \ref{pr.green}.

%

\subsection{Maximum of the integer-valued GFF.}\label{ss.max}
In this section, our goal is to prove Theorem \ref{th.maxIV} on the maximum of the integer-valued GFF.  The proof will be based on the following proposition on the Laplace transform of the IV-GFF (we also state a control on its Fourier transform which may useful elsewhere). Let us introduce the following shifted inverse temperature:
\begin{align}\label{e.tildeB}
\tilde \beta:= \left (1-M(\beta))\right ) \beta\,.
\end{align}

\begin{proposition}\label{pr.Lap_IV}
	Fix $\beta>0$ and let $\Psi$ be an IV-GFF with inverse-temperature $\beta^{-1}$ on the vertices of the a graph $\Lambda_n^0$ (i.e.  with zero boundary conditions). Then, for any $\bar \lambda>0$ and $\hat \beta >\tilde \beta$ there exist constants $K_1,K_2>0$ such that for all $\lambda\leq \bar \lambda$ and all $v\in \Lambda_n^0$
	\begin{align}
	\label{e.Laplace_IV}&\E^{IV}_{\beta^{-1}}\left[e^{\lambda\Psi(v)}\right]\leq K_1 e^{\frac{\lambda^2\hat \beta}{2} G_{\Lambda_n^0}(v,v)},\\
	\label{e.Fourier_IV}&\E^{IV}_{\beta^{-1}}\left[e^{i\lambda\Psi(v)}\right]\geq K_2 e^{-\frac{\lambda^2\hat \beta}{2} G_{\Lambda_n^0}(v,v)}. 
	\end{align}
\end{proposition}

\begin{remark}\label{r.imposible_FS}Let us note that Proposition \ref{pr.Lap_IV} cannot be generalised to any $0$-form $w$. In particular, there is no $\dot \beta <\beta, K>0$ such that for any $0$-form $w$ (or equivalently  $2$-form $w$)
	\begin{align}\label{e.false}
	\E_{\beta^{-1}}^{IV}\left[e^{\lambda\langle \Psi, w \rangle} \right] \leq K e^{\frac{\lambda^2\dot \beta}{2}\langle w ,\Delta^{-1} w \rangle}.
	\end{align}
	
	The main reason for this lies in \eqref{e.Vil_IV_int}. Let us take $\theta$ a Villain model on the vertices of $\Lambda_n$ and $\Psi$ and IV-GFF on the faces of $\Lambda_n$. Furthermore, fix a vertex $x \in \Lambda_n$ that is in the real line and take $\gamma$ the straight line from $0$ to $x$. Thanks to \eqref{e.Vil_IV_int}, we have that
	\begin{align*}
	\E^{Vil}_\beta\left[e^{i\langle \d \theta, \1_\gamma\rangle} \right]= \E^{IV}_{\beta^{-1}}\left[ e^{\frac{1}{\beta}\langle\d^*\Psi, \1_\gamma\rangle} \right] e^{-\frac{1}{2\beta}\langle \1_\gamma,\1_\gamma\rangle}=\E^{IV}_{\beta^{-1}}\left[ e^{\frac{1}{\beta}\langle\Psi, \d \1_\gamma\rangle } \right] e^{-\frac{1}{2\beta}\langle \1_\gamma,\1_\gamma\rangle}\,.
	\end{align*}
	Now, we note that
	\begin{align*}
	\langle \d \1_{\gamma}, \d \Delta^{-1} \1_{\gamma} \rangle &= \langle \1_{\gamma}, \1_{\Gamma} \rangle - \langle \d^* \1_{\gamma}, \Delta^{-1} \d^* \1_{\gamma} \rangle\\
	&= \| x\| - G_{\Lambda,v_0=0}(x,x) = \|x\| - \frac{1}{2\pi}\log(\|x\|) + o( \log(\|x\|).
	\end{align*}
	Here, we used that $\d^* \1_{\gamma}=\1_x-\1_0$.
	
	Assume now that \eqref{e.false} is true for $w=\d^* \1_{\gamma}$, 
	this would then lead us to 
	\begin{align*}
	\E_\beta^{Vil}\left[e^{i(\theta(x)-\theta(0))} \right]\leq K e^{\frac{\dot \beta}{2\beta ^2} ( \|x\| - \log(\|x\|))} e^{\frac{-\|x\|}{2\beta}}\leq K e^{-\frac{\epsilon}{2\beta} \|x\|}
	\end{align*}
	for any $\epsilon<1-\dot\beta/\beta$. This result would contradict the phase transition for the Villain model.
	
	This remark is the reason why we believe that it is not possible to obtain an upper bound of the Laplace transform of IV-GFF with the techniques of \cite{FS}. Their estimates  do not differentiate which test function $w$ we are working with. 
\end{remark}

\begin{proof}
	We start by using \eqref{e.Coul_GFF_2}. To agree with the notation there, we work with a GFF in the faces of the graph with free-boundary conditions, and we will identify the faces with the vertices of $\Lambda_n^0$. To prove \eqref{e.Laplace_IV}, we take $f$ a face and we use $g=\lambda\beta 1_{f}$ to obtain
	\begin{align*}
	\E^{IV}_{\beta^{-1}}\left[ e^{\lambda\langle\Psi, \1_{f}\rangle } \right]&= e^{\frac{\lambda^2\beta}{2}G_{\Lambda^*}(f,f)}\E^{Coul}_{\beta}\left[ e^{i2\pi \lambda \beta \langle \Delta^{-1}q,\1_{f} \rangle}  \right]\\
	&\leq Ke^{\frac{\lambda^2\beta}{2}G_{\Lambda^*}(f,f) (1- (1-\epsilon) M(\beta))}\,,
	\end{align*}
	where in the last inequality, we used \eqref{e.Fourier_Coul} with $f_1=\infty$ and $f_2= v$.
	
	The proof of \eqref{e.Fourier_IV} is analogue, we only need to use $g=i\lambda\beta 1_{v}$ and \eqref{e.Laplace_Coul} instead of \eqref{e.Fourier_Coul}. 
\end{proof}

Let us now rewrite Theorem \ref{th.maxIV} in the following explicit form using the shifted temperature $\tilde \beta$ defined above in~\eqref{e.tildeB}.
\begin{theorem}\label{th.maxIV2}
	Let $\Psi$ be an integer-valued GFF with inverse temperature $\beta^{-1}$ on the vertices of the graph $\Lambda_n^0$ (i.e.  with zero boundary conditions). We have that for any $\hat \beta> \tilde \beta$ 
	\begin{align*}
	\P^{IV}_{\beta^{-1}}\left[\sup_{v\in \Lambda_n^0}\Psi(v)\geq \frac{2}{\sqrt{2\pi \hat \beta ^{-1}}} \log(n) \right] \to 0, \ \ \text{as $n\to \infty$}.
	\end{align*}
\end{theorem}
The main step to prove this theorem is the following lemma. 
\begin{lemma}\label{l.tail_IV}
	Let $\Psi$ be a GFF at inverse temperature $\beta^{-1}$ on the vertices of the graph $\Lambda_n^0$ (i.e.  with zero boundary conditions). Then, we have that for all $\hat \beta>\tilde \beta$ and for all $ \alpha>0$, there exists $K>0$ such that
	\begin{equation}
	\P^{IV}_{\beta^{-1}}\left (\Psi(v)\geq \frac{\alpha}{\sqrt{ 2\pi \hat \beta^{-1}} }\log n \right )\leq K n^{-\frac{\alpha^2}{2}}, \ \ \text{ for all $n\in \N$ and $v\in \Lambda_n^0$}.
	\end{equation}
\end{lemma}
\begin{proof}
	Let us recall that there exists $K'>0$ such that uniformly on $n$, $G_{\Lambda_n^0}(v,v)$ is smaller than or equal to $(2\pi)^{-1} \log n+K'$. This implies that we can use the exponential Markov property to see that
	\begin{align*}
	\P^{IV}_{\beta^{-1}}\left (\Psi(v)\geq \frac{\alpha}{\sqrt {2\pi \hat \beta^{-1}}}\log n\right )&\leq \E_{\beta^{-1}}^{IV}\left[e^{\alpha \sqrt{2\pi\hat \beta^{-1}}\Psi(v)} \right]e^{-\alpha^2  \log n} \\
	&\leq Ke^{\frac{\alpha^22\pi \hat \beta^{-1}}{2\hat\beta^{-1}}G_{\Lambda_n^0}(v,v)}e^{-\alpha^2\log n}\\
	&\leq Ke^{\frac{\alpha^2}{2}2\pi K'} e^{-\frac{\alpha^2}{2}\log n}\\
	&\leq K''n^{\frac{-\alpha^2}{2}}.
	\end{align*}
	Note that in the second inequality we used Proposition \ref{pr.Lap_IV} with $\lambda= \alpha \sqrt{2\pi \hat \beta^{-1}}.$

\end{proof}

We can now prove Theorem \ref{th.maxIV2}.
\begin{proof}
	We use the union bound to see that for any $\beta'> \hat \beta >\beta$ we have that 
	\begin{align*}
	\P^{IV}_{\beta^{-1}}\left[\max_{v\in \Lambda_n^0}\Psi(v)\geq \frac{2\sqrt{\beta '}}{\sqrt{2\pi}}\log(n)  \right] &\leq \sum_{v\in \Lambda_n^0}\P^{IV}_{\beta^{-1}}\left[\Psi(v)\geq \frac{2\sqrt{\beta '}}{\sqrt{2\pi}}\log(n)  \right]\\
	&= \sum_{v\in \Lambda_n^0}\P^{IV}_{\beta^{-1}}\left[\Psi(v)\geq \frac{2\sqrt{\frac{\beta'}{\hat \beta}}}{\sqrt{2\pi \hat \beta^{-1}}}\log n \right]\\
	&\leq \sum_{v\in \Lambda_n^0} n^{-2\frac{\beta'}{\hat \beta}} \preceq n^{2(1-\frac{\beta'}{\hat \beta})} \to 0 \ \  \  \ \text{ as } n\to \infty, 
	\end{align*}
	where in the second inequality we used Lemma \ref{l.tail_IV} with $\alpha= 2\sqrt{\beta'/\hat \beta}$.\end{proof}

\appendix

\section{ReRooting GFF, Villain and Coulomb gas with free boundary conditions}\label{a.root}

When free boundary conditions are imposed around a domain $\Lambda$, there is no canonical choice for the rooting of 0-forms. In this work, we have chosen to fix a vertex $v_0$ on which 0-forms are set to be zero. The goal of this appendix is to analyse what is the effect of re-rooting to a different vertex $v_0'$ all the models that we encountered in this paper. These re-rooting are particularly convenient when we deal with two-point functions. For example Proposition \ref{pr.rootCoul} below on the re-rooting of a Coulomb gas was used in the proof of Propositions \ref{pr.Fourier_Coul} and \ref{pr.Laplace_Coul} on the Fourier and Laplace transforms of the Coulomb gas.
 
It may well be that the content of this appendix for the Coulomb gas will be considered {\em folklore} by the specialists, yet we did not find it written in the literature and we believe that despite its simplicity, it is likely to shed light on the rather peculiar properties of the Coulomb gas with {\em free boundary conditions} which has been studied extensively for example in \cite{FScoulomb,federbush1985}.

\subsection{Rerooting the GFF and $IV-$GFF.}

We start with the easiest case of  changing the marked point for the GFF.
\begin{proposition}
	Let $\phi$ be a GFF pinned at a point $v_0$, take $v'_0\in V$ and define
	\begin{align*}
	\phi'(\cdot)= \phi(\cdot)-\phi(v'_0).
	\end{align*}
	Then, $\phi'$ is a GFF pinned at the point $v_0'$.
\end{proposition}
\begin{proof}
	This follows directly from the fact that $\d \phi' = \d \phi$.
\end{proof}
Let us then discuss how to change the marked point for the integer-valued GFF.
\begin{proposition}
	Let $\Psi$ be an IV-GFF pinned at a point $v_0$, take $v'_0\in V$ and define
	\begin{align*}
	\Psi'(\cdot)= \Psi(\cdot)-\Psi(v'_0).
	\end{align*}
	Then, $\Psi'$ is a IV-GFF pinned at the point $v_0'$.
\end{proposition}
\begin{proof}
	This follows directly from the fact that $\d \Psi' = \d \Psi$, and from the fact that the function $\Psi\mapsto \Psi'$ is a bijection between integer-valued $0$-forms pinned in $v_0$ to integer-valued $0$-forms pinned in $v'_0$.
\end{proof}

\subsection{Rerooting the Coulomb gas.}

Let us now discuss how the law of a Coulomb gas changes when we change the marked point $v_0$.  (In this paper so far,  the Coulomb gas was considered on 2-forms, but in this appendix, we are not concerned with duality transformations, so we find it simpler to consider all three models to be defined on the vertices with free boundary conditions around $\Lambda$, this is why the rerooting of our Coulomb gas will be from a vertex $v_0$ to another $v_0'$. This translates directly when needed to a rerooting from a face $f_0$ to another $f_0'$).

\begin{proposition}\label{pr.rootCoul}
	Let $q$ be a Coulomb gas on the vertices pinned at a point $v_0$. Take $v_0'\in V$ and define
	\begin{align*}
	q'(v):= \left\{ \begin{array}{l l}
	q(v) & \text{ if }v\neq v_0 \text{ and }v\neq v_0',\\
	-\sum_{u\neq v_0} q(u)& \text{ if }v= v_0,\\
	0 & \text{ if }v=v_0'.
	\end{array}\right. 
	\end{align*}
	Then $q'(v)$ is a Coulomb gas on the vertices pinned at the point $v_0'$.
	
	The analogue of the result is also true for a Coulomb gas on the faces. Here a special orientation should be given for the faces (the same as in Remark \ref{r.no pinned}).
\end{proposition}
\begin{proof}
 Given that $q\mapsto q'$ is a bijection between integer-valued $0$-forms pinned in $v_0$ to integer-valued $0$-forms pinned in $v'_0$, we only need to prove that
 \begin{equation}
 \langle q , (-\Delta_{v_0})^{-1} q \rangle = \langle q' , (-\Delta_{v_0'})^{-1} q' \rangle.
 \end{equation}
 To do this, define 
 \begin{align*}
 \hat q(v):= \left\{ \begin{array}{l l}
 q(v) & \text{ if }v\neq v_0,\\
 -\sum_{u\neq v_0} q(u)& \text{ if }v= v_0.
 \end{array}\right. 
 \end{align*}
 Note that because $(-\Delta_{v_0})^{-1} q$ has value $0$ in the point $v_0$, we have that
 \begin{align}\label{e.coulomb_changing_v_0}
 \langle q, (-\Delta_{v_0})^{-1} q \rangle = \langle \hat q, (-\Delta_{v_0})^{-1} q \rangle= \left \langle \hat q, (-\Delta_{v_0})^{-1} q - (-\Delta_{v_0})^{-1} q (v_0') \right \rangle,
 \end{align}
 where the last equality holds because $\hat q$ has $0$-mean.
 
  Now, note that
 \begin{align*}
 (-\Delta_{v_0'})^{-1} \hat q(\cdot) =(-\Delta_{v_0'})^{-1}  q'(\cdot)=(-\Delta_{v_0})^{-1} q(\cdot) - (-\Delta_{v_0})^{-1} q (v_0').
 \end{align*}
 The first equality follows just because $\hat q= q'$ everywhere but in $v_0'$. The second equality follows by taking the Laplacian $\Delta$ on both sides and noting than the only complicated point may be $v_0$ (or $v_0'$ but the argument there will be the same), where we need that 
 \begin{align*}
 -\Delta(-\Delta_{v_0})^{-1}q (v_0)= q'(v_0) = -\sum_{u\neq v_0} q(u).
 \end{align*}
 This follows because for any $w$ a non-pinned $0$-form 
 \begin{align*}
 \sum_{v\in \Lambda} \Delta w(v)= \langle \d^* \d w, 1 \rangle =0.
 \end{align*}
 
 To finish, we continue \eqref{e.coulomb_changing_v_0}
 \begin{align*}
 \langle q, (-\Delta_{v_0})^{-1} q \rangle = \left \langle \hat q, (-\Delta_{v_0'})^{-1} q'\right \rangle=\left \langle  q', (-\Delta_{v_0'})^{-1} q'\right \rangle,
 \end{align*}
 where in the last equality we used that $(-\Delta_{v_0'})^{-1} q'$ takes value $0$ in the point $v_0'$.
 
\end{proof}

\subsection{Rerooting the Villain coupling $(\theta,m)$.}
To finish, we explain how to change the marked point $v_0$ for the Villain coupling.
\begin{proposition}
	Let $(\theta,m)$ be a Villain coupling pinned at a point $v_0$, take $v'_0\in V$ and define
	\begin{align*}
	&\theta'(\cdot):= \theta(\cdot)-\theta(v'_0) \mod 2\pi.\\
	&m'(e)=\left\{ \begin{array}{l l}
	m(e) & \text{if } \d \theta (e) = \d \theta' (e),\\
	m(e) - 1&\text{if } \d \theta (e) = \d \theta' (e) + 2\pi,\\
	m(e) + 1&\text{if } \d \theta (e) = \d \theta' (e) - 2\pi.
	\end{array}\right.
	\end{align*}
	Then, $(\theta',m')$ is a Villain coupling pinned at the point $v_0'$.
\end{proposition}
\begin{proof}
	This follows directly from the fact that $(\theta,m)\mapsto (\theta',m')$ is a bijection, and the fact that for all $(\theta,m)$ and all $e\in E$ one has that $\d \theta(e) + 2\pi m(e)=\d \theta '(e) + 2\pi m'(e)$.
\end{proof}

%

\section{Properties of the error function $M$.}\label{a.M}

The purpose of this (technical) appendix is to prove Proposition \ref{pr.Monotonicity_IG} and its consequences on the behaviour of the error function $M(\beta)$.

\begin{proof}[Proof of Proposition \ref{pr.Monotonicity_IG}]
Recall
\[
M(\beta) = (2\pi)^2\beta \inf_{a\in[0,1/2]}\var^{IG}(a,(2\pi)^2 \beta)\,.
\]
Our goal is to prove that for $\beta$ large enough ($\beta\geq \frac 1 3$), 
\begin{align*}\label{}
M(\beta) \geq 2\beta \exp(-\frac{(2\pi)^2} 2 \beta )\,.
\end{align*}

\ni
{\em To avoid carrying all way through $(2\pi)^2 \beta$-terms, we work with $\inf_{a\in[0,1/2]}\var^{IV}(a, \beta)$  and will get back to  $\beta\mapsto (2\pi)^2\beta$ at the very end.} 

\smallskip
\ni
\textbf{Step 1.}
First it can be checked that for any $a\in[0,\tfrac12]$, $\mu^{IG}(a,\beta)\in[0,\tfrac12]$. This is because $\mu^{IG}(a,\beta)$ is monotonous in $a$ and $\mu^{IG}(0,\beta)=0$ and $\mu^{IG}(\tfrac12,\beta)=\tfrac{1}{2}$. The monotony follow because if $0\leq a'<a \leq 1/2$
\begin{align*}
\frac{d \P^{IG}_{\beta,a}}{d\P^{IG}_{\beta,a'}}(x)\propto e^{\beta x (a-a')}.
\end{align*}
Thus,
\begin{align*}
\E^{IG}_{\beta,a}[X]= \frac{\E^{IG}_{\beta,a'}\left[ X e^{\beta X (a-a')} \right] }{\E^{IG}_{\beta,a'}\left[ e^{\beta X (a-a')} \right]}\geq \E^{IG}_{\beta,a'}[X],
\end{align*}
where in the last line we noted that both the identity function and $x\mapsto e^{\beta x(a-a')}$ are increasing function and we used Harris inequality.

\smallskip
\ni
\textbf{Step 2.} Notice first that for any $\beta$  large enough ($\beta\geq 10$ is enough),  
\begin{align}\label{e.mu-}
\var^{IG}(a,\beta)&= \frac{\sum_{n\in \Z} (n-\mu^{IG}(a,\beta))^2 e^{-\frac{\beta}{2}(n-a)^2}}{\sum_{n\in \Z} e^{-\frac{\beta}{2}(n-a)^2}} \nn \\
& \geq \frac{\mu^{IG}(a,\beta)^2 e^{-\beta \frac {a^2}{2}} + [\mu^{IG}(a,\beta)-1]^2  e^{-\beta \frac {(a-1)^2}{2}} }{2(e^{-\beta \frac {a^2}{2}}+e^{-\beta \frac {(a-1)^2}{2}})} \nn \\
& \geq \frac{[\mu^-(a)]^2 e^{-\beta \frac {a^2}{2}} + [\mu^+(a)-1]^2  e^{-\beta \frac {(a-1)^2}{2}} }{4 e^{-\beta \frac {a^2}{2}}}\,,
\end{align}
where for the numerator of the first inequality, we took the first two term of the sum and the denominator is bounded by the geometric sum. Now, for any choice of 
\begin{align*}\label{}
0\leq \mu^-(a) \leq \mu^{IG}(a,\beta)\leq \mu^+(a)\leq 1\,.
\end{align*}

\smallskip
\ni
\textbf{Step 3.}  Thus, we just need  to find suitable functions $a\mapsto \mu^-(a), \mu^+(a)$ which approximate well enough the mean $\mu^{IG}(a,\beta)$. Recall the latter one is given by 
\begin{align*}\label{}
\mu^{IG}(a,\beta)=\frac{\sum_{n\in \Z} n e^{-\frac{\beta}{2}(n-a)^2}}{\sum_{n\in \Z} e^{-\frac{\beta}{2}(n-a)^2}}\,.
\end{align*}
From this expression, we see that for any $\beta \geq 10$, say, one may choose:
\[
\begin{cases}
\mu^-(a)&= \frac
{e^{-\frac \beta 2 (1-a)^2} - e^{-\frac{\beta}2(1+a)^2}}
{4e^{-\frac\beta 2 a^2}} \\
\mu^+(a)&=\tfrac12
\end{cases}
\]
(N.B. there is no $O(e^{-\frac{\beta}2 (3/2)^2})$ correction term in $\mu^-(a)$ by using the fact $t\mapsto \exp(-\tfrac\beta 2 t^2)$ is non-increasing on $\R_+$, which implies by pairing integers in $\Z\setminus \{-1,0,1\}$ that their total contribution must be positive).
Plugging into~\eqref{e.mu-}, this gives us
\begin{align*}\label{}
\var^{IV}(a,\beta)& \geq \frac 1 {32} e^{-\beta } (\sinh(\beta a))^2 + \frac 1 {16} e^{-\frac \beta 2 (1-2a)} \geq \frac 1 {16} e^{-\frac \beta 2(1-2a)}\,.
\end{align*}
\end{proof}
\begin{remark}\label{}
Note that for this lower bound, it would have been sufficient to take $\mu^-(a):=0$. We included this more refined lower bound in case one would want wish to improve the Lemma into an equivalent of $M(\beta)$. Indeed, the contribution coming from our choice of $\mu^-(a)$ highlights what happens when $a\to 0$.  
\end{remark}

\begin{proof}[Proof of Corollary \ref{c.M}]
By changing $\beta \mapsto (2\pi)^2\beta$, this gives us for any $\beta \geq \frac{10}{(2\pi)^2}$ 
\begin{align*}\label{}
M(\beta)& = (2\pi)^2\beta \inf_{a\in[0,1/2]}\var^{IV}(a,(2\pi)^2 \beta) \\
& \geq  \frac{(2\pi)^2\beta}{16} e^{-\frac {(2\pi)^2} 2 \beta }
\geq 2 \beta   e^{-\frac {(2\pi)^2} 2 \beta }
\end{align*}
The fact that $M(\beta)\asymp \beta   e^{-\frac {(2\pi)^2} 2 \beta }$ now follows from: 
\begin{align*}\label{}
M(\beta)& \leq M_0(\beta) \sim 2(2\pi)^2 \beta e^{-\frac {(2\pi)^2} 2 \beta }
\end{align*}
\end{proof}
We end this appendix with the proof of Lemma \ref{l.ratioK}

\ni
{\em Proof of Lemma \ref{l.ratioK}.}

Recall that our goal is to show that

	\begin{align*}
	K_\beta=
	\sup_{\hat \beta \geq \beta} \sup_{a\in \R} \frac{T^{IG}(a,\hat \beta)}{\var^{IG}(a,\hat \beta)}\in(0,\infty)
	\end{align*}
Let us note that 
	\begin{align*}
	K_\beta(a):= \frac{T^{IG}(a,\beta)}{\var^{IG}(a,\beta)}
	\end{align*}
	is a continuous function for $a\in \R$ and $\beta \in (0,\infty)$, furthermore $K_\beta(a)=K_\beta(a+k)$ for all $k\in \Z$ and $K_\beta(a)=K_\beta(1-a)$. Thus, we just need to show that as $\beta \to \infty$
	\begin{align*}\label{e.lim_sup}
	\limsup_{\beta \nearrow \infty} \sup_{a\in [0,1/2]} \frac{T^{IG}(a,\beta)}{\var^{IG}(a,\beta)}  \leq 1.
	\end{align*}
	
	We start by noting that for all $(a,\beta)\in (0,1/2)\times \R^+$ 
	\begin{align*}
	0\leq \mu(a,\beta) \leq \frac{1}{2}.
	\end{align*}
	Then, let us take $\beta$ big enough so that for all $a\in [0,1/2]$
	\begin{align*}
	\sum_{n\in \Z\backslash\{-1,0,1\}}|n-a|^3e^{-\frac{\beta}{2} (n-a)^2}< e^{-\beta}.
	\end{align*}
	This can be done thanks to the fact that $(3/2)^2>2$. We can, now, use that $|k-\mu(a,\beta)|<2$ to compute
	\begin{align*}
	\frac{T^{IG}(a,\beta)}{\var^{IG}(a,\beta)}&\leq  \frac{\sum_{k=-1}^1 |k-\mu(a,\beta)|^3 e^{-\frac{\beta}{2}(n-a)^2 } + e^{-\beta}}{\sum_{k=-1}^1 (k-\mu(a,\beta))^2 e^{-\frac{\beta}{2}(n-a)^2 }}\\
	&\leq  2+ \frac{e^{-\beta}}{(1-\mu(a,\beta))^2 e^{-\beta(1-a)^2 }}\\
	&\leq 2 + 4 e^{-\frac{3\beta}{4} }.
	\end{align*}
	This proves what we needed. \qed

\section{Small gradients for the Villain model at low temperature via Reflection Positivity}\label{a.RP}

The purpose of this appendix is to prove the very intuitive statement that at low temperature, spins in any given finite window tend to align in the same direction. This looks like a very simple and legitimate property, yet we did not find a short and direct proof of this ``small gradient'' property and our proof below relies on reflection positivity and more precisely on the co-called {\em chessboard} estimate.  
We refer the reader to the following useful references on the notion and use of reflection positivity \cite{FSS1976,FILS1978,FrohlichLieb1978,biskupRP,velenikBook}. 
See also the references \cite{bricmont1981correlation,cohen2020rarity}.

\begin{proposition}\label{pr.SmallGr}
For any $\eps>0$ and any $R>0$, there exists $\beta_0=\beta_0(\eps,R)<\infty$ such that the following holds: for any $\beta\geq \beta_0$,  there exists $L_0\in\N$ such that  for any $n\geq L_0$, if one considers the Villain model on $\Lambda_n$ with either free or Dirichlet boundary conditions, then for all $x\in \Lambda_n$ at distance at least $L_0$ from $\p \Lambda_n$ one has
\begin{align*}
\FK{\beta}{Vil}{\d \theta(e)\in(-\eps,\eps) \text{ for all }e\in \Lambda_R+x} \geq 1-\eps\,.
\end{align*}
\end{proposition}


The proof will be divided in the following three steps which will be the content of the next three subsections:
\bnum
\item As we did not find a clean proof that Villain's model is indeed reflection positive, we provide a short sketch of proof. 
\item We then prove that gradients are small on the two-dimensional torus using the chessboard estimate.
\item Finally we explain how to recover Proposition \ref{pr.SmallGr} from the case of the torus by going through infinite volume limits. 
\enum

\subsection{Reflection positivity of the Villain interaction.}
In the rest of this appendix $\T_n$ will denote the two-dimensional torus $\Z^2/n\Z^2$. 
\begin{lemma}\label{}
For any $n\geq 1$, and any inverse temperature $\beta$, the Villain model on $\T_n$ is reflection positive. 
\end{lemma}
\ni
{\em Proof (sketch).}
If for any plane of reflection $P$ in $\T_n$ with corresponding reflection operator $\Theta$, one may write the torus Hamiltonian as 
\[
-H_n= A + \Theta A + \sum_\alpha C_\alpha \Theta C_\alpha\,,
\]  
where $A$ and $\{C_\alpha\}_\alpha$ are functions which are all defined on the same side of the torus w.r.t. $P$, then the Gibbs measure $e^{-\beta H_n(\{\sigma_x\})}\prod_{x\in \T_n} d\mu(\sigma_x)$ is a reflection positive measure (where $\mu$ denotes here the uniform measure on $\mathbb{S}^1$). 
See for example the very useful references \cite{biskupRP,velenikBook} for this way of obtaining reflection positive measure. 

This family of reflection positive measure easily extends to Gibbs measures whose Hamiltonians are of the form 
\[
-H_n= A + \Theta A + \int_{\alpha \in K} C_\alpha \Theta C_\alpha d\alpha\,,
\]  
for any continuous function from a compact set $K$ to the set of functions on one side the torus. 

Let us then rewrite the Villain measure $\P_\beta^{Vil}$ as follows,
\begin{align*}\label{}
\P_\beta^{Vil}(\{\theta_x\})
& = \frac 1 {Z_\beta^{Vil}} \prod_{e=\{x,y\} \in \T_n} \sum_{m} \exp(-\tfrac \beta 2(\theta_x-\theta_y + 2\pi m)^2)\\
&= \frac 1 {Z_\beta^{Vil}} \exp\left( \sum_{e=\{x,y\} \in \T_n}  \log \sum_{m} \exp(-\tfrac \beta 2(\theta_x-\theta_y + 2\pi m)^2) \right)
\end{align*}
Following the notations of \cite{biskupRP,velenikBook}, denote by $\T_n^+,\T_n^-$ the two sides of $\T_n$ w.r.t the plane $P$. Let us deal with the case where $P$ does not intersect any vertices (the other case can be analyzed similarly). 
Let then 
\[
A:= \sum_{e=\{x,y\} \subset \T_n^+}  \log \sum_{m} \exp(-\tfrac \beta 2(\theta_x-\theta_y + 2\pi m)^2)\,,
\]
so that $A +\Theta A$ takes into account all nearest-neighbor interactions besides the $1D$ line of edges $e=\{x,y\}$ with $x\in T_n^-$ and $y\in T_n^+$. For any such transverse edge, notice that if $t:=\beta^{-1}$, 
\[
\sum_{m\in \Z}\exp(-\tfrac \beta 2(\theta_x-\theta_y + 2\pi m)^2) = \sqrt{2\pi t} \mathbf p_{t}^{\S^1}(\theta_x,\theta_y)\,,
\]
where $\mathbf p_{t}^{\S^1}(x,y)$ denotes the heat kernel on the unit circle $\S^1 \simeq [0,2\pi)$. Using the Markov property at time $s=\tfrac t2=\tfrac 1 {2\beta}$, we may thus write 
\[
\sum_{m\in \Z} \exp(-\tfrac \beta 2(\theta_x-\theta_y + 2\pi m)^2) = \sqrt{2\pi t} \int_{u\in \S^1}  \mathbf p_{t/2}^{\S^1}(\theta_y, u) \mathbf p_{t/2}^{\S^1}(\theta_2,u) du\,.
\]
Notice furthermore that for any fixed $t>0$, $\inf_{\theta}{\mathbf p_{t}^{\S^1}(\theta,0)}>0$ which allows us 
to express $\log \sum_{m\in \Z}\exp(-\tfrac \beta 2(\theta_x-\theta_y + 2\pi m)^2)$ as the convergent series 
\begin{align*}\label{}
& \sum_{k\geq 1} \tfrac{(-1)^k}k  \left(\int_{u\in \S^1} \sqrt{2\pi t} \mathbf p_{t/2}^{\S^1}(\theta_x, u) \mathbf p_{t/2}^{\S^1}(\theta_y,u) du  - 1 \right)^k \\
& = \sum_{k\geq 1} \tfrac{1 }k \sum_{m=0}^k \binom{k}{m}(- \sqrt{2\pi t})^{m} \int_{u_1\in S^1,\ldots, u_m\in \S^1}
\prod_{i=1}^m p_{t/2}^{\S^1}(\theta_x, u_i)p_{t/2}^{\S^1}(\theta_y, u_i) d u_i\,.
\end{align*}
We have thus rewritten the Villain model with an Hamiltonian written in the form 
\[
-H_n= A + \Theta A + \sum_{\ell \geq 1} \int_{\alpha=(\theta_1,\ldots, \ell)\in K_\ell:=(\S^1)^\ell} C_\alpha \Theta C_\alpha \; \mu_{K_\ell}(d\alpha)\, 
\]
where the sum is convergent. This proves that the Villain measure is reflection positive. 
\qed

\subsection{Small gradients on the torus via the Chessboard estimate.}

\begin{proposition}\label{pr.SmallGrTorus}
For any $\delta>0$, there exists $\beta_0$ s.t. for any $n\geq 1$ any $\beta\geq \beta_0$ and any edge $e=\{x,y\}$,
\begin{align*}
\FK{\beta,\T_n}{Vil}{\d \theta(e)\in(-\delta,\delta)} \geq 1- \delta\,,
\end{align*}
where $\d \theta$ is understood here modulo $2\pi$.
(N.B. The proof below shows that $\beta_0$ can be chosen to be $12 \log(8/\delta) \delta^{-2}$).
\end{proposition}

\ni
{\em Proof.}
We apply the chessboard estimate (\cite{FrohlichLieb1978,biskupRP,velenikBook} to the $1\times 2$ (or $2\times 1$ depending on the orientation of $e$) block $B:=\{x,y\}$ and to the local $B$-block function $f(\{\theta\}):=1_{|\theta_x-\theta_y|\geq\delta}$ (where $\theta_x-\theta_y$ is also understood modulo $2\pi$). As one can cover $\T_n$ with successive reflections of the block $B$, the chessboard estimate gives us if $\beta$ is large enough,
\begin{align*}
\FK{\beta,\T_n}{Vil}{\d \theta(e)\notin(-\delta,\delta)} & \leq \left(\FK{\beta,\T_n}{Vil}{\text{all, say, horizontal edges $f$ are s.t. } \d \theta(f) \geq \delta} \right)^{2/n^2} \\
& \leq \left( \frac 1 {Z^{Vil}_{\beta}} \int_{(\S^1)^{\T_n}} \exp(\sum_{x\in \T_n} \log \exp(-\tfrac \beta 3 \delta^2) + \log 2)\right)^{2/n^2} \\
& \leq \left( \frac 1 {Z^{Vil}_{\beta}} \int_{(\S^1)^{\T_n}} \exp( - \tfrac \beta 4 \delta^2 n^2) \right)^{2/n^2}\,.
\end{align*}

To conclude, for any $\alpha$, note that 
\begin{align*}\label{}
Z_\beta^{Vil}&:= \int_{(\S^1)^{\T_n}} \prod_{e} \sum_{m\in \Z}\exp(-\tfrac \beta 2 (d\theta(e) + 2\pi m)^2 ) \prod_{x\in\T_n} d\theta_x \\
& \geq 
\exp(-\tfrac \beta 2 \alpha^2 *(2 n^2)) \int_{(\S^1)^{\T_n}} 1_{\{\theta_x\in[-\alpha,\alpha], \forall x \in \T_n\}} \prod_{x\in\T_n} d\theta_x \\
& = \exp(-\tfrac \beta 2 \alpha^2 *(2 n^2))  \alpha^{n^2}\,.
\end{align*}
This gives us the following bound
\begin{align*}
\FK{\beta,\T_n}{Vil}{\d \theta(e)\notin(-\delta,\delta)} 
& \leq \left(  \exp( - \beta (\tfrac 1 4\delta^2 - 2\alpha^2) n^2  +\log(1/\alpha) n^2)  \right)^{2/n^2} \\
& \leq \exp\left(- \beta ( \tfrac 1 2 \delta^2 - 4\alpha^2) + 2 \log(1/\alpha)\right)\,.
\end{align*}
Choosing for example $\alpha:=\tfrac \delta 4$, we obtain the quantitative bound,
\begin{align*}
\FK{\beta,\T_n}{Vil}{\d \theta(e)\notin(-\delta,\delta)} \leq 
\exp\left(- \beta \tfrac{\delta^2}4   + 2 \log(4/\delta)\right)\,,
\end{align*}
which proves the lemma for any $\beta\geq \beta_0=12 \log(8/\delta) \delta^{-2}$. 
\qed

\subsection{Back to finite domains.}

\begin{lemma}
$ $
\bi
\item[i)](Villain model with Dirichlet boundary conditions on $\Lambda_n$). For any $\delta>0$, $n\in \N$ any $\beta\geq \hat\beta_0=\hat \beta_0(\beta)$ and any edge $e\in \Lambda_n$, 
\begin{align*}\label{}
\FK{\beta,\Lambda_n^0}{Vil}{\d \theta(e)\in(-\delta,\delta)} \geq 1- \delta\,.
\end{align*}
\item[ii)](Villain model with free boundary conditions on $\Lambda_n$). For any $\delta>0$ and any $\beta\geq \hat \beta_0 = \hat \beta_0(\beta)$ (again the same as in Proposition \ref{pr.SmallGrTorus}) there exists $M_0=M_0(\beta)\in \N$ such that for any $n\geq M_0$ and any edge $e\in \Lambda_n$ at distance at least $M_0$ from $\p \Lambda_n$, 
\begin{align*}\label{}
\FK{\beta,\Lambda_n^{free}}{Vil}{\d \theta(e)\in(-2\delta,2\delta)} \geq 1- 2\delta\,.
\end{align*}
\ei
\end{lemma}
Using $\delta:=\eps/(10R^2)$, $L_0:=M_0+R$, it is straightforward to deduce Proposition \ref{pr.SmallGr} from the above Lemma and using that
	\begin{align*}
&\FK{\beta}{Vil}{\d \theta(e)\in(-\eps,\eps) \text{ for all }e\in \Lambda_R+x} \\
&\geq\P_\beta^{Vil}\left[ \d \theta(e)\in(-\frac{\eps}{10R^2},\frac{\eps}{10R^2}) \text{ for all }e\in \Lambda_R+x\right] \\
&\geq 1-\sum_{e\in \Lambda_R+x} \frac{\eps}{10R^2}\\
&\geq 1-\epsilon.
	\end{align*}
We are thus left with the proof of the Lemma.

\smallskip
\ni
{\em Proof.}

\ni
For the case $i)$ of Dirichlet boundary conditions, by using Ginibre correlation inequality (\cite{Ginibre}) as explained for example in \cite{messager1978}, 
it follows readily that gradients are smaller under Dirichlet conditions than on the Torus. I.e. for any edge $e\in \Lambda_n$,
\begin{align*}\label{}
\EFK{\beta,\Lambda_n^0}{Vil}{\cos(\d \theta(e))} \geq 
\EFK{\beta,\T_n}{Vil}{\cos(\d \theta(e))}\,.
\end{align*}
Note that we need here that Ginibre's inequality applies also to the case of the Villain's interaction. This is indeed the case, see Section 3. in \cite{FScoulomb}.  
By choosing $\hat \delta$ sufficiently small and $\beta \geq \hat \beta_0:=\beta_0(\hat \delta)$ in Proposition \ref{pr.SmallGrTorus}, this implies
\begin{align*}\label{}
\EFK{\beta,\Lambda_n^0}{Vil}{\cos(\d \theta(e))} \geq 
\EFK{\beta,\T_n}{Vil}{\cos(\d \theta(e))} \geq 1 - \tfrac13 \hat \delta^2\,, 
\end{align*}
which in turn implies (if $\hat \delta$ is chosen small enough) $\FK{\beta,\Lambda_n^0}{Vil}{\d \theta(e)\in(-\delta,\delta)}\geq 1-\delta$.  
\smallskip
\ni
For the case $ii)$ with {\em free} boundary conditions, we use the following result from \cite{messager1978} (where again, as explained in \cite{FScoulomb}, we use the fact that such results for the XY model transfer to the Villain interaction). 
\begin{theorem}[\cite{messager1978}]\label{}
For any $\beta>0$, there is a unique translation invariant Gibbs measure on $\P_{\beta,\Z^2}^{Vil}$ on $(\S^1)^{\Z^2}$. Furthermore, $\P_{\beta,\Lambda_n^{free}}^{Vil}$, $\P_{\beta,\Lambda_n^{0}}^{Vil}$ and $\P_{\beta,\T_n}^{Vil}$ all converge to $\P_{\beta,\Z^2}^{Vil}$ as $n\to\infty$. 
\end{theorem}
This implies that (if $e_0$ is any edge around the origin), 
\begin{align*}\label{}
\liminf_{M\to \infty} \FK{\beta,\Lambda_M^{free}}{Vil}{\d \theta(e_0)\in (-2\delta,2\delta)} \geq \limsup_{n\to \infty} \FK{\beta,\T_n}{Vil}{\d \theta(e)\in (-\delta,\delta)} \geq 1-\delta\,. 
\end{align*}
\qed

\bibliographystyle{alpha}
\bibliography{biblio}

\end{document}